\documentclass[12pt]{amsart}

\usepackage{amssymb}

\usepackage{enumitem}

\usepackage{graphicx}

\makeatletter
\@namedef{subjclassname@2010}{%
  \textup{2010} Mathematics Subject Classification}
\makeatother

\usepackage[T1]{fontenc}
\newtheorem{theorem}{Theorem}[section]
\newtheorem{corollary}[theorem]{Corollary}
\newtheorem{lemma}[theorem]{Lemma}

\newtheorem{claim}[theorem]{Claim}

\newtheorem{mainthm}[theorem]{Main Theorem}

\theoremstyle{definition}
\newtheorem{definition}[theorem]{Definition}
\newtheorem{remark}[theorem]{Remark}

\newtheorem{fact}[theorem]{Fact}

\numberwithin{equation}{section}

\begin{document}

%%%%% To ease editing, for IMPAN journals add:

\baselineskip=17pt

%%%%%%%%%%%%%%%%

\title[Permutation Groups]{Pseudofinite primitive permutation groups acting on one-dimensional sets}

\author[T. Zou]{Tingxiang Zou}
\address{Institut Camille Jordan\\
Universit\'{e} Lyon I \\
21, avenue Claude Bernard\\
69622 Villeurbanne, France}
\email{zou@math.univ-lyon1.fr}

\date{}

\begin{abstract}
Working in a theory with an integer-valued dimension on interpretable sets, we classify pseudofinite definably primitive permutation groups acting on one-dimensional sets which satisfy a version of chain condition on centralizers and on point-wise stabilizers. This generalises the classification of pseudofinite definably primitive permutation groups acting on a rank $1$ set in supersimple theories of finite rank in \cite{Elwes} to supersimple theories of infinite rank.
\end{abstract}



\maketitle

\section{Introduction}

Let $G$ be a transitive permutation group acting on a set $X$. 
We say that $(G,X)$ is \textsl{pseudofinite} if it is elementarily equivalent to a non-principal ultraproduct of finite permutation groups. We might as well assume that $(G,X)=\prod_{i\in I}(G_i,X_i)/\mathcal{U}$ for some non-principal ultrafilter $\mathcal{U}$ on an infinite set $I$.

Finite primitive permutation groups have been classified into several types by the well-known O'Nan-Scott Theorem. This classification reduces most problems concerning finite primitive permutation groups to problems of finite simple groups. Together with the Classification of Finite Simple Groups, it gives a good understanding of finite primitive permutation groups. As pseudofinite groups can be seen as limits of finite groups, we might wonder if it is also possible to give a nice description of pseudofinite permutation groups. There have been some attempts. In \cite{Liebeck}, pseudofinite definably primitive permutation groups have been extensively studied via the O'Nan-Scott Theorem. In \cite{Elwes}, under the additional assumption that $(G,X)$ lives in a supersimple theory of finite $SU$-rank and that the $SU$-rank of $X$ is one, Elwes, Jaligot, Macpherson and Ryten managed to get a complete classification, which is analogous to the well-known classification of stable permutation groups acting on strongly minimal sets in \cite{hrushovski1989almost}.

We specify the language for permutation groups: $\mathcal{L}$ contains two sorts $G$ and $X$, with the group language $\{\cdot,(-)^{-1},id\}$ on $G$ and a function $(-)^{(-)}:X\times G\to X$ which represents the action of $G$ on $X$.

We recall the classification in \cite{Elwes}.
\begin{fact}\label{Elwes-main} (\cite[Theorem 1.3]{Elwes})

Let $(G,X)$ be a pseudofinite definably primitive permutation group. Let $T$ be the theory of $(G,X)$ in the language $\mathcal{L}$. Suppose $T$ is supersimple of finite $SU$-rank such that $T^{eq}$ eliminates $\exists^\infty$ and $SU(X)=1$. Then the socle of $G$ (the subgroup generated by all minimal non-trivial normal subgroups), $soc(G)$, exists and is definable, and one of the following holds:
\begin{enumerate}
\item
$SU(G)=1$, and $soc(G)$ is abelian of finite index in $G$ and acts regularly on $X$;
\item
$SU(G)=2$, and there is an interpretable pseudofinite field $F$ of $SU$-rank $1$ such that $(G,X)$ is definably isomorphic to $(F^+\rtimes H,F^+)$, where $H\leq F^\times$ is of finite index.
\item
$SU(G)=3$, and there is an interpretable pseudofinite field $F$ of $SU$-rank 1 such that $(G,X)$ is definably isomorphic to $(H,\text{PG}_1(F))$, where $\text{PSL}_2(F)\leq H\leq \text{P}\Gamma\text{L}_2(F)$.\footnote{In fact, we think $H$ should be contained in $\text{PGL}_2(F)$, there shouldn't be any non-trivial automorphism of $F$ induced by $G$, see Lemma \ref{lem-Btrivial} and Corollary \ref{thmnonsimple}.} Moreover, $soc(G)$ is definably isomorphic to $\text{PSL}_2(F)$.
\end{enumerate}
\end{fact}

This result is based on the investigation of pseudofinite groups of small $SU$-rank in the same paper \cite{Elwes}. Basically, they showed that pseudofinite groups of $SU$-rank 1 are finite-by-abelian-by-finite, and those of $SU$-rank 2 are soluble-by-finite. We list them here.

\begin{fact}\label{fact-finitebyabelian}(\cite[Lemma 3.1(i)]{Elwes}) Let $G$ be an infinite group definable in a supersimple theory $T$ such that $T^{eq}$ eliminates $\exists^{\infty}$. Let $H\leq G$ be an infinite finite-by-abelian subgroup. Then $H$ is contained in an infinite definable finite-by-abelian subgroup $K\leq G$.
\end{fact}

\begin{fact}\label{fact-solublebyfinite}(\cite[Theorem 1.2]{Elwes}) Let $G$ be a pseudofinite group definable in a supersimple theory $T$ such that $T^{eq}$ eliminates $\exists^{\infty}$. Suppose $SU(G)=2$. Then $G$ is soluble-by-finite.
\end{fact}

The analysis of pseudofinite groups of small $SU$-rank has been generalised in \cite{Wagner} to a wider context which includes the pseudofinite supersimple and superrosy groups of infinite rank. Basically, Wagner replaces finite $SU$-rank by an abstract dimension which satisfies some nice properties, together with some chain condition on centralizers. 

Model-theoretically tame theories can often be viewed or defined in more than one way. For example, we can define tame theories as those who have a well-behaved independence relation, which is often attained by forking independence. One of the other definitions requires a well-behaved dimension, for example the Morley-rank, Lascar-rank, $SU$-rank and so on. Interestingly, the existence of a nice independence relation or dimension is often related to abstract combinatoric properties that a theory should not be able to define, for example, the independence property or the strict order property.

The generalization in \cite{Wagner} tries to capture model-theoretic tameness in a more abstract and unified way. 
The aim of introducing an abstract dimension is to unify several different dimension-like objects in tame theories, for example the Lascar or $SU$-rank in stable and simple theories, the o-minimality dimension and the pseudofinite counting dimension. On the other hand, the chain condition on centralizers focuses more on the combinatoric properties that a tame theory should have. This condition itself decreases the complexity of groups and gives some nice structural theorems for definable subgroups (see \cite{Hempel} and \cite{Hempel16} for more details). However, classical tame model theory usually has more powerful well-developed tools for analysis, for example the Indecomposability Theorem in supersimple theories. It is extensively used in \cite{Elwes}. We state the version for supersimple finite $SU$-rank groups here.

\begin{fact}\label{fact-indecomp}(Indecomposability Theorem, \cite[Theorem 5.4.5]{Wagner-Supersimple}) Let G be a group definable in a supersimple finite $SU$-rank theory and $\{X_i:i\in I\}$ be a (possibly infinite) collection of definable subsets of $G$. Then there exists a definable subgroup $H$ of $G$ such that
\begin{enumerate}
\item
$H\leq \langle X_i, i\in I\rangle$, and there are $i_0,\ldots,i_n\in I$ such that $H\leq X_{i_0}^{\pm 1}\cdots X_{i_n}^{\pm 1}$;
\item
$X_i/H$ is finite for each $i\in I$.
\end{enumerate}
Moreover, if the collection $\{X_i:i\in I\}$ is setwise invariant under some group $\Sigma$ of definable automorphisms of $G$, then $H$ can be chosen to be $\Sigma$-invariant.
\end{fact}

In this paper, we generalize Fact \ref{Elwes-main} to the same context as in \cite{Wagner}, which in particular includes the pseudofinite definably primitive permutation groups in supersimple or superrosy theories of infinite rank. Interestingly, as we do not assume supersimplicity of the ambient theory, the Indecomposability Theorem is not available. However, in one main step of the proof, we go to a subgroup of the permutation group, whose theory in the pure group language is supersimple. Via this, we use the powerful structural theorems in supersimple theories to get the desired result.

Let us introduce the general context that we will work with and state our main theorem. 

\begin{definition}
A \textsl{dimension} on a theory $T$ is a function $\mathsf{dim}$ from all non-empty interpretable subsets of a monster model to $\mathbb{R}^{\geq 0}\cup\{\infty\}$, satisfying:
\begin{enumerate}
\item
Invariance: If $a\equiv a'$, then $\mathsf{dim}(\varphi(x,a))=\mathsf{dim}(\varphi(x,a'))$;

\item
Algebraicity: If $X$ is finite, then $\mathsf{dim}(X)=0$;

\item
Union: $\mathsf{dim}(X\cup Y)=\textsf{max}\{\mathsf{dim}(X),\mathsf{dim}(Y)\}$;

\item
Fibration: If $f:X\to Y$ is a surjective interpretable function and $\mathsf{dim}(f^{-1}(y))\geq r$ for all $y\in Y$, then $\mathsf{dim}(X)\geq\mathsf{dim}(Y)+r$;
\end{enumerate}

We define the dimension of a tuple of elements $a$ over a set $B$ as $$\mathsf{dim}(a/B):=\inf\{\mathsf{dim}(\varphi(x)):\varphi\in \text{tp}(a/B)\}.$$

When the equation $\mathsf{dim}(a,b/C)=\mathsf{dim(a/b,C)+\mathsf{dim}(b/C)}$ holds for any tuples $a,b$ and any set $C$, we say that the dimension $\mathsf{dim}$ is \textsl{additive}.

When $\mathsf{dim}$ has its range in $\mathbb{N}$ then we say that the dimension $\mathsf{dim}$ is \textsl{integer-valued}.

\end{definition}

\begin{remark}\label{remark-dim}
In our definition, if a dimension is integer-valued, then it cannot take the value $\infty$. Note this is different with the definition in \cite{Wagner}.

In pseudofinite structures there is a class of counting dimensions, called \emph{coarse dimensions}. They satisfy all the conditions for the dimension we defined above (possibly in an expansion of the language to insure invariance). They are additive (in a certain expansion of the language), but not necessarily integer-valued.

Another family of examples of dimensions is the following. Take a superstable (or supersimple, or superrosy) theory, suppose $rk(T)=\omega^\alpha\cdot n+\beta$ for some ordinals $\alpha,\beta$ with $\beta<\omega^{\alpha}$ and some integer $n$, where $rk$ is Lascar, $SU$ or thorn-rank. Then for any interpretable set $X$, define $\mathsf{dim}(X):=k$ if $rk(X)=\omega^{\alpha}\cdot k+\gamma$ for some $k\in\mathbb{N}$ and $\gamma<\omega^\alpha$. With this definition, $\mathsf{dim}$ is an additive integer-valued dimension.

Note that in the definition of a dimension, it is not required that dimensional $0$ sets are finite. In fact, in the examples above where the dimension comes from the coefficient of $\omega^\alpha$ of Lascar/SU/thorn-rank with $\alpha\neq 0$, we will always have infinite definable sets of dimension $0$. This is one of the major difficulties in generalizing Fact \ref{Elwes-main}, \ref{fact-finitebyabelian} and \ref{fact-solublebyfinite}.
\end{remark}

\begin{definition}
Let $G$ be a group. We say that $G$ satisfies the \textsl{$\widetilde{\mathfrak{M}_c}$-condition} or $G$ is an \textsl{$\widetilde{\mathfrak{M}_c}$-group} if the following holds: 
 $$\exists d\in\mathbb{N}, \forall g_0,\cdots,g_d\in G,~\bigvee_{i<d}\left([C_G(g_0,\cdots,g_i): C_G(g_0,\cdots,g_{i+1})]\leq d\right).$$
\end{definition}

\begin{remark}\label{remark-mc}  By \cite[Theorem 4.2.12, Proposition 4.4.3]{Wagner-Supersimple}, all interpretable groups in simple theories satisfy the $\widetilde{\mathfrak{M}_c}$-condition. 
\end{remark}

Here is the generalization of Fact \ref{fact-finitebyabelian} and \ref{fact-solublebyfinite} in \cite{Wagner}.

\begin{fact}\label{fact-frank-finitebyabelian}(\cite[Theorem 4.11, Corollary 4.14]{Wagner}) Let $G$ be a pseudofinite $\widetilde{\mathfrak{M}_c}$-group with an additive dimension $\mathsf{dim}$ such that $\mathsf{dim}(G)>0$. Then the following holds.
\begin{enumerate}
\item
$G$ has a definable finite-by-abelian subgroup $C$ with $\mathsf{dim}(C)>0$.
\item
If $\mathsf{dim}$ is integer-valued and $\mathsf{dim}(G)=1$, then $G$ has a definable characteristic finite-by-abelian subgroup $C$ such that $\mathsf{dim}(C)=1$.
\end{enumerate} 
\end{fact}

\begin{fact}\label{fact-frank-soluble}(\cite[Theorem 5.1, Corollary 5.2]{Wagner}) Let $G$ be a pseudofinite $\widetilde{\mathfrak{M}_c}$-group with an additive integer-valued dimension $\mathsf{dim}$ such that $\mathsf{dim}(G)=2$. Then the following holds.
\begin{enumerate}
\item
$G$ has a definable finite-by-abelian subgroup $C$ such that $\mathsf{dim}(C)\geq 1$ and $\mathsf{dim}(N_G(C))=2$.
\item
If definable sections of $G$ also satisfy the $\widetilde{\mathfrak{M}_c}$-condition, then $G$ has a definable soluble subgroup $D$ with $\mathsf{dim}(D)=2$.
\end{enumerate}
\end{fact}

\begin{remark} The proof of Fact \ref{fact-frank-soluble}, more precisely, of Theorem 5.1 in \cite{Wagner} uses the Classification of Finite Simple Groups (CFSG). But the assumption of Theorem 5.1 in \cite{Wagner} is slightly weaker than the one we stated. We refer to an earlier version of this proof, \cite[Theorem 13, Corollary 14]{Mcgroups}, which does not use the CFSG. 
\end{remark}

We recall the definition of a (definably) primitive permutation group.
\begin{definition}
A permutation group $G$ acting on a non-empty set $X$ is called \textsl{primitive} if $G$ acts transitively on $X$ and preserves no non-trivial partition of $X$. If $G$ is transitive and preserves no non-trivial definable partition of $X$, then $G$ is called \textsl{definably primitive}.
\end{definition}

\begin{remark}
A transitive permutation group $G$ is primitive if and only if any point stabilizer $\mathsf{Stab}_G(x):=\{g\in G: x^g=x\}$ is a maximal proper subgroup of $G$. Similarly, $G$ is definably primitive if and only if any $\mathsf{Stab}_G(x)$ is a definably maximal proper subgroup of $G$, that is there is no definable subgroup $D\leq G$ such that $\mathsf{Stab}_G(x)\lneq D\lneq G$.
\end{remark}

\begin{definition}
We define $\mathcal{S}$ to be the class of all pseudofinite definably primitive permutation groups $(G,X)$ with an additive integer-valued dimension $\mathsf{dim}$ such that $\mathsf{dim}(X)=1$, and such that $G$ satisfies the $\widetilde{\mathfrak{M}_c}$-condition.
\end{definition}

By Remark \ref{remark-dim} and Remark \ref{remark-mc}, $\mathcal{S}$ contains all pseudofinite definably primitive permutation groups $(G,X)$ in supersimple finite $SU$-rank theories such that $SU(X)=1$. The  aim of this paper is to get a classification of $\mathcal{S}$ similar to Fact \ref{Elwes-main}. It turned out that the restrictions on $\mathcal{S}$ are enough for us to classify members of $\mathcal{S}$ of dimension 1 and 2. We need more combinatorial assumptions for dimension greater or equal to 3, one of which is similar to the $\widetilde{\mathfrak{M}_c}$-condition but for stabilizers, and the other one is a minimality condition on $X$. . We list them here.

Notation: Let $G$ be a group acting on some structure $X$, for $x\in X$ we write $\mathsf{Stab}_G(x)$ for the point-stabilizer $\{g\in G:x^g=x\}$, and for $B\subseteq X$ we write $$\mathsf{PStab}_G(B):=\bigcap_{x\in B}\mathsf{Stab}_G(x)$$ as the point-wise stabilizer.

\begin{enumerate}
\item
$\widetilde{\mathfrak{M}_s}$-condition on $(G,X)$: There is $d\in\mathbb{N}$ such that
$$\forall x_0,\ldots,x_d\in X,\bigvee_{i<d}\left([\mathsf{PStab}_G(x_0,\ldots,x_i): \mathsf{PStab}_G(x_0,\ldots,x_{i+1})]\leq d\right).$$
\item
(EX)-condition on $X$:\\
$X$ contains no infinite set of 1-dimensional equivalence classes for any definable equivalence relation on $X$.
\end{enumerate}

\begin{remark}\label{remark-ms}  As the $\widetilde{\mathfrak{M}_c}$-condition, all interpretable groups in simple theories also satisfy the $\widetilde{\mathfrak{M}_s}$-condition, by \cite[Theorem 4.2.12, Proposition 4.4.3]{Wagner-Supersimple}.
\end{remark}

Now we are able to state our main result.

\begin{mainthm}
Let $(G,X)\in\mathcal{S}$.
\begin{enumerate}
\item
If $\mathsf{dim}(G)=1$, then $G$ has a definable normal abelian subgroup $A$, such that $\mathsf{dim}(A)=1$ and $A$ acts regularly on $X$.
\item
If $\mathsf{dim}(G)= 2$ and definable sections of $G$ satisfy the $\widetilde{\mathfrak{M}}_c$-condition. Then there is a definable subgroup $H\trianglelefteq G$ of dimension $2$, and an interpretable pseudofinite field $F$ of dimension $1$, such that $(H,X)$ is definably isomorphic to $(F^+\rtimes D,F^+)$ for some definable $D\leq F^\times$ of dimension $1$.
\item
If $\mathsf{dim}(G)\geq 3$. Suppose definable sections of $G$ satisfy the $\widetilde{\mathfrak{M}}_c$-condition, $G$ satisfies the $\widetilde{\mathfrak{M}}_s$-condition and $X$ satisfies the (EX)-condition. Then $\mathsf{dim}(G)= 3$ and there is a definable subgroup $D\leq G$ of dimension $3$ and an interpretable pseudofinite field $F$ of dimension $1$ such that $D$ is definably isomorphic to $\textup{PSL}_2(F)$ and $(G,X)$ is definably isomorphic to $(H, \textup{PG}_1(F))$, where $\textup{PSL}_2(F)\leq H\leq \textup{P}\Gamma\textup{L}_2(F)$.
\end{enumerate}

\end{mainthm}

The Main Theorem enables us to analyse the pseudofinite definably primitive permutation groups of infinite $SU$-rank, which is an immediate generalization of Fact \ref{Elwes-main}.

\begin{corollary}
Let $(G,X)$ be a pseudofinite definably primitive permutation group in a supersimple theory. Suppose $SU(G)=\omega^\alpha n+\gamma$ and $SU(X)=\omega^\alpha+\beta$ for some $\gamma,\beta<\omega^\alpha$ and $n\in\mathbb{N}$. Then one of the following holds:
\begin{enumerate}
\item
$SU(G)= \omega^\alpha+\gamma$, and there is a definable abelian subgroup $A$ of $SU$-rank $\omega^\alpha$ acting regularly on $X$. 
\item
$SU(G)=2$, and there is an interpretable pseudofinite field $F$ of $SU$-rank $1$ with $(G,X)$ definably isomorphic to $(F^+\rtimes H,F^+)$, where $H$ is a subgroup of $F^\times$ of finite index.
\item
$SU(G)=3$, and there is an interpretable pseudofinite field $F$ of $SU$-rank $1$ with $(G,X)$ definably isomorphic to $(\textup{PSL}_2(F), \textup{PG}_1(F))$ or $(\textup{PGL}_2(F),\textup{PG}_1(F))$.
\end{enumerate}
\end{corollary}

\begin{remark} Fact \ref{Elwes-main} uses the Classification of Finite Simple Groups for $SU$-rank greater or equal to 3, so do our results for dimension greater or equal to 3, in particular Section \ref{sec5} and Section \ref{sec6} uses the CFSG without mentioning it explicitly. 
\end{remark}

The rest of this paper is organised as the following. Section \ref{sec2} gives some general analysis of the basic properties of $\widetilde{\mathfrak{M}_c}$-groups with an additive integer-valued dimension. Section \ref{sec4} deals with pseudofinite definably primitive permutation groups of dimensions 1 and 2. The main results are Theorem \ref{thm-dim1} and Theorem \ref{thm-dim2}. Section \ref{sec5} handles the rest, i.e., permutation groups of dimension greater or equal to 3. The corresponding result is obtained in Theorem \ref{thm-dim3}. The last part, section \ref{sec6} studies the special case of pseudofinite definably primitive permutation groups in supersimple theories of infinite rank. Theorem \ref{thm-noInfiniteRank} concludes this section.

\section{$\widetilde{\mathfrak{M}_c}$-Groups with a Dimension}\label{sec2}
In this section we will first establish some  general results about $\widetilde{\mathfrak{M}_c}$-groups with an additive integer-valued dimension.

In the following lemmas, we assume that $\mathsf{dim}$ is an additive integer-valued dimension on a group $G$.

\begin{definition}
We say a subgroup $H\leq G$ is \textsl{broad} if $\mathsf{dim}(H)>0$. And we say $H$ is \textsl{wide in $G$} if $\mathsf{dim}(H)=\mathsf{dim}(G)$.
\end{definition}

\begin{lemma}\label{claim 1}
Let $H_0,\ldots,H_n$ be a finite family of wide definable subgroups of $G$. Then $\bigcap_{i\leq n}H_i$ is also wide in $G$.
\end{lemma}
\begin{proof}
It suffices to prove the claim when $n=1$, the rest follows by induction. By the properties of dimension, we have that $\mathsf{dim}(G/H_0)=\mathsf{dim}(G)-\mathsf{dim}(H_0)=0$. Similarly, $\mathsf{dim}(G/H_1)=0$. 

Note that there is a definable injection from $G/(H_0\cap H_1)$ to $G/H_0\times G/H_1$ sending $g(H_0\cap H_1)$ to $(gH_0,gH_1)$. Hence $\mathsf{dim}(G/(H_0\cap H_1))\leq \mathsf{dim}(G/H_0)+\mathsf{dim}(G/H_1)= 0$. We obtain $$\mathsf{dim}(H_0\cap H_1)=\mathsf{dim}(G)-\mathsf{dim}(G/(H_0\cap H_1))=\mathsf{dim}(G).\qedhere$$
\end{proof}

\begin{lemma}\label{claim 1.0}
Suppose $G$ is finite-by-abelian. Then for any $g_0,\ldots,g_n\in G$, the centralizer $C_G(g_0,\ldots,g_n)$ is wide in $G$.
\end{lemma}
\begin{proof}
Since $G$ is finite-by-abelian, the derived subgroup $G'$ is finite. For any $g\in G$, the set $g^{-1}g^G=\{g^{-1}h^{-1}gh:h\in G\}$ is a subset of $G'$, hence is finite. Therefore, $g^G$ is finite and is of dimension 0. Note that there is a definable bijection between $g^G$ and $G/C_G(g)$. Thus, $\mathsf{dim}(C_G(g))=\mathsf{dim}(G)-\mathsf{dim}(g^G)=\mathsf{dim}(G)$.
As $C_G(g_i)$ is definable and wide in $G$ for each $i\leq n$, so is $C_G(g_0,\ldots,g_n)$ by Lemma \ref{claim 1}.
\end{proof}

\begin{lemma}\label{claim 5}
Let $B_1\trianglelefteq A_1$ and $B_2\trianglelefteq A_2$ be subgroups of $G$. If both $A_1/B_1$ and $A_2/B_2$ are finite-by-abelian, then so is $(A_1\cap A_2)/(B_1\cap B_2)$. 
\end{lemma}
\begin{proof}
For the derived subgroups, we have 
\begin{align*}
((A_1\cap A_2)/(B_1\cap B_2))'&=((A_1\cap A_2)'(B_1\cap B_2))/(B_1\cap B_2)\\
&\subseteq ((A_1'\cap  A_2')(B_1\cap B_2))/(B_1\cap B_2).
\end{align*} 
Since both $A_1'B_1/B_1=(A_1/B_1)'$ and $A_2'B_2/B_2=(A_2/B_2)'$ are finite, so is the product $(A_1'B_1/B_1)\times (A_2'B_2/B_2)$. Define a function $$f:((A_1'\cap  A_2')(B_1\cap B_2))/(B_1\cap B_2)\longrightarrow (A_1'B_1/B_1)\times (A_2'B_2/B_2)$$ by sending $a(B_1\cap B_2)$ to $(aB_1,aB_2)$. It is easy to check that $f$ is injective. Therefore, $((A_1'\cap  A_2')(B_1\cap B_2))/(B_1\cap B_2)$ is finite. We conclude that $((A_1\cap A_2)/(B_1\cap B_2))'$ is finite and $(A_1\cap A_2)/(B_1\cap B_2)$ is finite-by-abelian.
\end{proof}

From now on, we assume further that $G$ is $\widetilde{\mathfrak{M}_c}$.

\begin{definition}
Let $H_1$ and $H_2$ be two subgroups  of $G$. We say $H_1$ is \textsl{almost contained} in $H_2$, denoted as $H_1\lesssim H_2$, if $[H_1:H_2\cap H_1]<\infty$. If both $H_1\lesssim H_2$ and $H_2\lesssim H_1$ hold, then $H_1$ and $H_2$ are called \textsl{commensurable}. 

For two subgroups $H,K\leq G$, the \textsl{almost centralizer} of $K$ in $H$ is defined as $$\widetilde{C}_H(K):=\{h\in H:[K:C_K(h)]<\infty\}.$$ The \textsl{almost center} is defined as $\widetilde{Z}(H):=\widetilde{C}_H(H)$.

Let $\mathcal{D}$ be an infinite family of subgroups of $G$. We say $\mathcal{D}$ is \textsl{uniformly commensurable} if there is some $N\in\mathbb{N}$ such that $[D:D\cap D']\leq N$ for all $D,D'\in \mathcal{D}$.
\end{definition}

\begin{remark}\label{remark-Z(G)}
When $G$ is $\widetilde{\mathfrak{M}_c}$ and $H,K$ are definable subgroups of $G$, then $\widetilde{C}_H(K)$ is also definable. (\cite[Proposition 3.28 ]{Hempel16})
\end{remark}

We list a useful fact for almost centralizers here.

\begin{fact}\label{fact-centralizer}\cite[Theorem 3.13]{Hempel16}
Let $H$ and $K$ be two definable subgroups of $G$. Then $H\lesssim \widetilde{C}_G(K)$ if and only if $K\lesssim \widetilde{C}_G(H)$.
\end{fact}

\begin{lemma}\label{claim 2}
Let $D:=C_G(\bar{g})$ be the centralizer of some finite tuple $\bar{g}
\in G^n$. Suppose $D$ is wide in $G$. Then there is a wide definable normal subgroup $N$ of $G$ such that $N$ is commensurable with $E:=\bigcap_{i\leq k}D^{t_i}$ for some $k\in\mathbb{N}$ and $t_0,\ldots,t_k\in G$.
\end{lemma}
\begin{proof}
By the $\widetilde{\mathfrak{M}_c}$-condition, there are $t_0,\ldots, t_k\in G$ and $d\in \mathbb{N}$ such that for any $t\in G$ we have $[\bigcap_{i\leq k}D^{t_i}:\bigcap_{i\leq k}D^{t_i}\cap D^t]\leq d$. Let $E:=\bigcap_{i\leq k}D^{t_i}$. Since $E$ is a finite intersection of wide subgroups, $E$ is also wide by Lemma \ref{claim 1}. For any $h_1,h_2\in G$, $$[E^{h_1}:E^{h_1}\cap E^{h_2}]=[E:E\cap E^{h_2 h^{-1}_1}]\leq \prod_{i\leq k}[E:E\cap D^{t_ih_2h_1^{-1}}]\leq d^{k+1}.$$ Therefore $\mathcal{E}:=\{E^t:t\in G\}$ is a family of uniformly commensurable definable subgroups of $G$. By Schlichting's Theorem (\cite[Theorem 4.2.4]{Wagner-Supersimple}), there is a definable subgroup $N$ of $G$, which is invariant under all automorphisms of $G$ stabilizing $\mathcal{E}$ setwise, and is commensurable with all members of $\mathcal{E}$. In particular, $N$ is normal in $G$ and is commensurable with $E$, hence is also wide.
\end{proof}

\begin{lemma}\label{claim 4}
Let $M,N$ be subgroups of $G$. Then $$\widetilde{Z}(M)\cap\widetilde{Z}(N)\leq \widetilde{Z}(M)\cap N\leq \widetilde{Z}(M\cap N).$$
\end{lemma}
\begin{proof}
Clearly, we have $\widetilde{Z}(M)\cap\widetilde{Z}(N)\leq \widetilde{Z}(M)\cap N$ for any $M,N\leq G$.

If $g\in \widetilde{Z}(M)\cap N$, then $g\in M\cap N$ and $[M:C_{M}(g)]<\infty$. Hence, $$[M\cap N:C_{M\cap N}(g)]=[M\cap N:C_M(g)\cap N]\leq [M:C_M(g)]<\infty,$$ and we get $g\in \widetilde{Z}(M\cap N)$. Therefore, $\widetilde{Z}(M)\cap N\leq \widetilde{Z}(M\cap N)$.
\end{proof}

\begin{lemma}\label{claim 3}
Let $M,N$ be subgroups of $G$. If $M$ is commensurable with $N$, then $\widetilde{Z}(M)$ is commensurable with $\widetilde{Z}(N)$. 
\end{lemma}
\begin{proof}
If $g\in \widetilde{Z}(M\cap N)$, then $$[M:C_M(g)]\leq [M:C_{M\cap N}(g)]\leq[M:M\cap N][M\cap N:C_{M\cap N}(g)]<\infty,$$ hence, $g\in\widetilde{Z}(M)$. Similarly, $\widetilde{Z}(M\cap N)\leq \widetilde{Z}(N)$. Therefore, $\widetilde{Z}(M\cap N)\leq\widetilde{Z}(M)\cap\widetilde{Z}(N)$. Together with Lemma \ref{claim 4}, we have $$\widetilde{Z}(M\cap N)=\widetilde{Z}(M)\cap\widetilde{Z}(N)=\widetilde{Z}(M)\cap N=\widetilde{Z}(N)\cap M.$$

Since $M,N$ are commensurable, $$[\widetilde{Z}(M):\widetilde{Z}(M)\cap \widetilde{Z}(N)]=[\widetilde{Z}(M):\widetilde{Z}(M)\cap N]\leq[M:M\cap N]<\infty.$$ Similarly, $\widetilde{Z}(N)$ and $\widetilde{Z}(M)\cap \widetilde{Z}(N)$ are commensurable. 
%  By claim \ref{claim 4}, $\widetilde{Z}(M)\cap N\leq\widetilde{Z}(M\cap N)$. Thus $[\widetilde{Z}(M):\widetilde{Z}(M\cap N)]\leq [\widetilde{Z}(M):\widetilde{Z}(M)\cap N]<\infty$. 
%Therefore, $[\widetilde{Z}(M):\widetilde{Z}(M)\cap\widetilde{Z}(N)]\leq[\widetilde{Z}(M):\widetilde{Z}(M\cap N)]<\infty$. By symmetry, we also have $[\widetilde{Z}(N):\widetilde{Z}(M)\cap\widetilde{Z}(N)]<\infty$.
\end{proof}

\begin{lemma}\label{claim-0dimension}
Let $H,D$ be definable subgroups of $G$. Define $$H^D_0:=\{h\in H,~\mathsf{dim}(h^D)=0\}.$$ Then there are $d\in\mathbb{N}$ and a definable group $T\leq D$ such that $$H^D_0=\{h\in H,~[T:C_T(h)]\leq d\}.$$ In particular, $H^D_0$ is a definable subgroup of $H$.
\end{lemma}
\begin{proof}
It is easy to see that $\mathsf{1}\in H^D_0$ and that it is closed under inverse. Note that $(h_1h_2)^D\subseteq h_1^Dh_2^D$. Therefore, if $h_1,h_2\in H^D_0$, then $$\mathsf{dim}((h_1h_2)^D)\leq \mathsf{dim}(h_1^D)+\mathsf{dim}(h_2^D)=0.$$ Hence, $h_1h_2\in H^D_0$.

By the $\widetilde{\mathfrak{M}_c}$-condition, there are $h_0,\cdots,h_n\in H^D_0$ and $d\in\mathbb{N}$ such that $[T:C_T(h)]\leq d$ for all $h\in H^D_0$, where $T:=C_D(h_0,\cdots,h_n)$. Since for each $h_i$, $\mathsf{dim}(C_D(h_i))=\mathsf{dim}(D)$, we have $\mathsf{dim}(T)=\mathsf{dim}(C_D(h_0,\cdots,h_n))=\mathsf{dim}(D)$. Let $$M:=\{h\in H, ~[T:C_T(h)]\leq d\}.$$ Then $M$ is definable. We claim that $M=H^D_0$. By definition, $H^D_0\subseteq M$. On the other hand, if $h\in M$, then $\mathsf{dim}(C_D(h))\geq \mathsf{dim}(C_T(h))=\mathsf{dim}(T)=\mathsf{dim}(D)$. Hence, $\mathsf{dim}(h^D)=0$ and $h\in H^D_0$. 
\end{proof}
\medskip

\section{Permutation Groups of Dimension 1 and 2}\label{sec4}
In this section, we analyse the permutation groups in $\mathcal{S}$ of dimension 1 or 2.

Here is a useful lemma for (definably) primitive permutation groups that we will use a lot without referring to it explicitly.

\begin{lemma}\label{lem-trivialtansitive} Let $(G,X)$ be a (definably) primitive permutation group and $A$ a (definable) normal subgroup of $G$. Then $A$ is either trivial or acts transitively on $X$.
\end{lemma}
The proof is standard, we leave it to the readers.

\begin{lemma}\label{claim 5.1}
Let $(G,X)$ be a definably primitive permutation group. If $G$ has a definable non-trivial normal abelian subgroup $A$, then $A$ acts regularly on $X$ and $A$ is either divisible torsion free or elementary abelian. 

Moreover, $G=A\rtimes G_{x}$ where $G_{x}=\mathsf{Stab}_G(x)$ for some $x\in X$, and $G_{x}$ acts on $X=x^A\simeq A$ by conjugation.

In particular if $(G,X)\in\mathcal{S}$, then we have in addition $\mathsf{dim}(A)=1$.
\end{lemma}
\begin{proof}
As $G$ acts definably primitively on $X$ and $A\trianglelefteq G$ is non-trivial, $A$ acts transitively on $X$. If $x^a=x^b$ for some $x\in X$ and $a,b\in A$, then for any $y\in X$, by transitivity, $y=x^c$ for some $c\in A$. As $A$ is abelian, we get $$y^a=x^{ca}=x^{ac}=x^{bc}=x^{cb}=y^b.$$ Hence, $a=b$. Therefore, $A$ acts regularly on $X$. Fix some $x\in X$. Then $a\mapsto x^a$ is a definable bijection from $A$ to $X$. Thus, if $(G,X)\in\mathcal{S}$, then $\mathsf{dim}(A)=\mathsf{dim}(X)=1$.

For any $n\in\omega$ let $nA:=\{a^n:a\in A\}$. Then $nA$ is a definable characteristic subgroup of $A$, hence definable abelian normal in $G$. If $\mathsf{dim}(nA)=1$, then $nA$ also acts regularly on $X$, whence $nA=A$. Otherwise, $\mathsf{dim}(nA)=0$, and $nA$ is trivial by definable primitivity of $G$. Therefore, $A$ is either divisible torsion free or elementary abelian. 

Let $G_{x}:=\mathsf{Stab}_G(x)$. As $A$ acts regularly on $X$, we have $A\cap G_{x}=\{\mathsf{1}\}$. For any $g\in G$ there is a unique element $a\in A$ such that $x^a=x^g$. Hence, $x=x^{ga^{-1}}$, so $ga^{-1}\in G_{x}$ and $g\in AG_{x}$. As $A\cap G_{x}=\{\mathsf{1}\}$, we obtain $G=A\rtimes G_{x}$.
%Therefore, we can write $g$ as $(gag^{-1})(ga^{-1})$ which is inside $A\rtimes G_{x_0}$. If $g=a'b$ with $a'\in A,b\in G_{x_0}$, then $(x_0)^g=(x_0)^{a'b}=(x_0)^{b(b^{-1}a'b)}=(x_0)^{b^{-1}a'b}$. Hence $a=b^{-1}a'b$ and $a'=bab^{-1}$. Since $g=a'b=ba$, we get $b=ga^{-1}$. Consequently, $a'=gag^{-1}$. Therefore, $G$ is the semi-direct product of $A$ and $G_{x_0}$.

Note that for any $g\in G_{x}$ and any $a\in A$, we have $(x^a)^g=x^{g^{-1}ag}$. Therefore, if we identify $A$ with $X$ via $a\mapsto x^a$, then $G_{x}$ acts on $A$ by conjugation.
\end{proof}

Combining the two lemmas above, we get the first part of our main result.
\begin{theorem}\label{thm-dim1}
Let $(G,X)\in\mathcal{S}$. If $\mathsf{dim}(G)=1$, then $G$ has a definable wide abelian normal subgroup $A$ such that $A$ acts regularly on $X$. Moreover, $A$ is either divisible torsion-free or elementary abelian. 
\end{theorem} 
\begin{proof}
By Fact \ref{fact-frank-finitebyabelian}(2), $G$ has a definable wide normal finite-by-abelian subgroup $A$. Consider the derived subgroup $A'$. It is finite and characteristic in $A$, hence is a definable normal subgroup of $G$. Since $G$ acts definably primitively on $X$, either $A'$ is trivial or $A'$ acts transitively on $X$. If $A'$ acts transitively on $X$, then $\mathsf{dim}(A')\geq\mathsf{dim}(X)=1$, contradicting that $A'$ is finite. Hence $A'$ is trivial and $A$ is a definable wide abelian normal subgroup of $G$. By Lemma \ref{claim 5.1}, $A$ acts regularly on $X$ and is either divisible torsion free or elementary abelian.
\end{proof}

We now proceed to analyse the groups in $\mathcal{S}$ of dimension greater than 1. The following lemma gives a key property of them.

\begin{lemma}\label{lem3.0}
Let $(G,X)\in\mathcal{S}$ with $\mathsf{dim}(G)\geq 2$. If $K\trianglelefteq G$ and $\mathsf{dim}(K)\geq 2$, then there is no element $a\in K\setminus\{\mathsf{1}\}$, such that $C_K(a)$ is wide in $K$.
\end{lemma}
\begin{proof}
Suppose, towards a contradiction, that there is $a\in K\setminus\{\mathsf{1}\}$ and $\mathsf{dim}(C_K(a))=\mathsf{dim}(K)\geq 2$. By the $\widetilde{\mathfrak{M}_c}$-condition, there are $g_0,\cdots,g_n\in G$ such that $$\lbrace(\bigcap_{i\leq n}C_K(a^{g_i}))^g:g\in G\rbrace$$ is a uniformly commensurable family. Since $K\trianglelefteq G$, we have $a^{g_i}\in K$ and $(\bigcap_{i\leq n}C_K(a^{g_i}))^g$ is a subgroup of $K$ for any $g\in G$. Note that $C_{K}(a^{g_i})=(C_{K}(a))^{g_i}$ is wide in $K$ for each $g_i$. Thus, $\mathsf{dim}(\bigcap_{i\leq n}C_K(a^{g_i}))=\mathsf{dim}(K)\geq 2$.

By Schlichting's Theorem there is a definable subgroup $N$ of $K$ such that $N\trianglelefteq G$ and is commensurable with $\bigcap_{i\leq n}C_K(a^{g_i})$, whence wide in $K$. Consider the group $\widetilde{Z}(N)$. We claim that $\mathsf{dim}(\widetilde{Z}(N))\geq 1$. Since $N$ is commensurable with $\bigcap_{i\leq n}C_K(a^{g_i})$, we have $a^{g_i}\in \widetilde{C}_K(N)$ and $a^{g_i}\neq\mathsf{1}$. As $\widetilde{C}_K(N)$ is definable normal in $G$, by definable primitivity of $G$, it is of dimension at least 1 (otherwise, it would be trivial). Note that $\widetilde{Z}(N)=N\cap \widetilde{C}_K(N)$. Then 
\begin{align*}
\mathsf{dim}(\widetilde{Z}(N))&=\mathsf{dim}(K)-\mathsf{dim}(K/\widetilde{Z}(N))\\
&\geq \mathsf{dim}(K)-(\mathsf{dim}(K/N)+\mathsf{dim}(K/\widetilde{C}_K(N)))
\\
&\geq \mathsf{dim}(K)-0-\mathsf{dim}(K)+\mathsf{dim}(\widetilde{C}_K(N))=\mathsf{dim}(\widetilde{C}_K(N))\geq 1.
\end{align*}
Therefore $\widetilde{Z}(N)$ acts transitively on $X$.

By \cite[Proposition 4.23]{Hempel16}, the commutator group $E:=[\widetilde{Z}(N),\widetilde{C}_N(\widetilde{Z}(N))]$ is finite. Since $N$ is normal in $G$ and $E$ is characteristic in $N$ and definable of dimension zero, $E$ is trivial. Therefore, $\widetilde{C}_N(\widetilde{Z}(N))\subseteq C_N(\widetilde{Z}(N))$.
 
We claim that $\widetilde{C}_N(\widetilde{Z}(N))$ is wide in $K$. Indeed, by Fact \ref{fact-centralizer}, we have $N\lesssim \widetilde{C}_N(\widetilde{Z}(N))$ if and only if $\widetilde{Z}(N)\lesssim \widetilde{C}_N(N)=\widetilde{Z}(N)$. Thus, $N$ is commensurable with $\widetilde{C}_N(\widetilde{Z}(N))$.
 
Let $H:=C_N(\widetilde{Z}(N))$. Then $H$ is a definable wide subgroup of $K$ and is normal in $G$. Fix $x\in X$. For all $h\in \widetilde{Z}(N)$, $$\ \mathsf{Stab}_H(x^h)=(\mathsf{Stab}_H(x))^h=\mathsf{Stab}_H(x).$$  Since $\widetilde{Z}(N)$ acts transitively on $X$, we get $\mathsf{Stab}_H(x)=\{\mathsf{1}\}$. However, by the Orbit-Stabilizer Theorem, $$\mathsf{dim}(\mathsf{Stab}_{H}(x))=\mathsf{dim}(H)-\mathsf{dim}(\mathsf{Orb_H(x)})=\mathsf{dim}(K)-\mathsf{dim}(X)\geq 2- 1=1,$$ contradicting that $\mathsf{Stab}_H(x)=\{\mathsf{1}\}$.
\end{proof}

In the following, we will show that if we have a finite-by-abelian group acting on a one-dimensional abelian group, then under certain conditions, we can define a pseudofinite field.

\begin{theorem}\label{thm2D}
Let $A$ be an abelian group of dimension 1 and $D$ a broad %finite-by-ablian 
definable group of automorphisms of $A$. Suppose that $A_0\leq A$ is definable of dimension 0 and $D$ acts on $A/A_0$. Let $D_0:=\{d\in D:\forall a\in A,~a^d\in a+A_0 \}$, a definable normal subgroup of $D$. Write $a+A_0\in A/A_0$ as $[a]$ and $dD_0\in D/D_0$ as $[d]$. Suppose $D$ satisfies the following condition:
\centerline{
($\clubsuit$) $\mathsf{dim}([a]^{C_{D/D_0}([d_1],\ldots,[d_n])})=1$ for any $[a]\neq [\mathsf{0}]$ and $n\in\mathbb{N}, d_1,\ldots,d_n\in D$.}

Then there is an interpretable pseudofinite field $F$ such that $F^+$ is isomorphic to $A/A_0$ and $D/D_0$ embeds into $F^\times$ with $\mathsf{dim}(D/D_0)=1$.
\end{theorem}
Remark: If $D$ is finite-by-abelian and $A_0:=\{a\in A: \mathsf{dim}(a^D)=0\}$ is of dimension 0, then condition $(\clubsuit)$ is satisfied. Indeed, $C_D(d_1,\ldots,d_n)$ has finite index in $D$ when $D$ is finite-by-abelian. As $a\not\in A_0$ by assumption, $\mathsf{dim}(a^D)=1$. Hence, $\mathsf{dim}(a^{C_D(d_1,\ldots,d_n)})=\mathsf{dim}(a^D)=1$ and $$\mathsf{dim}([a]^{[C_D(d_1,\ldots,d_n)]})=\mathsf{dim}([a]^{C_{D/D_0}([d_1],\ldots,[d_n])})=1.$$ Also note that condition $(\clubsuit)$ implies that $\mathsf{dim}(a^D)=1$ for $a\not\in A_0$.

Let $\mathcal{R}_D(A/A_0)$ be the ring of endomorphisms of $A/A_0$ generated by $D$, with addition being the component-wise addition on $A$ and multiplication being composition. Then any $r\in \mathcal{R}_D(A/A_0)$ is equal to some $\sum_{i\leq n}(-1)^{\epsilon_i} d_i$, but this representation need not be unique.

%We write the element $a+A_0$ of $A/A_0$ as $[a]$. Note that for any $d\in D$ and $a\in A$ we have $[a]^d=[a^d]$. For a subset $B\subseteq A$ we also write $(B+A_0)/A_0\subseteq A/A_0$ as $[B]$.

\begin{lemma}\label{lem4.0}
For all $r\in\mathcal{R}_D(A/A_0)$, either $r$ is the constant $[\mathsf{0}]$ function $\mathbf{0}$, or $r$ is an automorphism of $A/A_0$.
\end{lemma}
\begin{proof}
%Let $D_0:=\{d\in D:\forall a\in A,~a^d\in a+A_0\}$, a definable normal subgroup of $D$ since $A_0$ is $D$-invariant. We write $hD_0\in D/D_0$ as $[h]$.

We first prove the following claim: if there is some $[a]\in A/A_0$ such that $[a]\neq[\mathsf{0}]$ and $[a]^r=[\mathsf{0}]$, then $\mathsf{dim}(\mbox{ker}(r))=1$. Indeed, let $d_1,\ldots,d_n$ be the elements of $D$ which appear in a representation of $r$. Then $([a]^{[h]})^{r}=([a]^{r})^{[h]}=[\mathsf{0}]$ for any $[h]\in C_{D/D_0}([d_1],\ldots,[d_n])$. As a consequence, $[a]^{C_{D/D_0}([d_1],\ldots,[d_n])}\subseteq \mbox{ker}(r)$. We have  $\mathsf{dim}([a]^{C_{D/D_0}([d_1],\ldots,[d_n])})=1$ by condition $(\clubsuit)$. Therefore, $\mbox{ker}(r)$ has dimension 1.

Now we prove a similar assertion for the dimension of the image: if there is some $[a]\neq[\mathsf{0}]$ such that $[a]^r\neq[\mathsf{0}]$, then $\mathsf{dim}(\mbox{im}(r))=1$. Let $d_1,\ldots,d_n$ be all the elements in $D$ which appear in a representation of $r$. For any $[d]\in C_{D/D_0}([d_1],\ldots,[d_n])$, we have $([a]^{[d]})^{r}=([a]^{r})^{[d]}$, i.e., $([a]^r)^{[d]}\in \mbox{im}(r)$. Hence, $([a]^r)^{C_{D/D_0}([d_1],\ldots,[d_n])}\subseteq \mbox{im}(r)$. Then $$1\geq \mathsf{dim}(\mbox{im}(r))\geq \mathsf{dim}(([a]^r)^{C_{D/D_0}([d_1],\ldots,[d_n])})=1.$$ 

Since $\mathsf{dim}(\mbox{ker}(r))+\mathsf{dim}(\mbox{im}(r))=\mathsf{dim}(A/A_0)=1$, we can conclude that either $\mbox{ker}(r)=\{[\mathsf{0}]\}$ or $\mbox{im}(r)=\{[\mathsf{0}]\}$. If $\mbox{im}(r)=\{[\mathsf{0}]\}$, then $r=\mathbf{0}$. Otherwise $r$ is injective. As $(G,X)$ is a pseudofinite structure, $r$ must also be surjective, hence an automorphism. 
\end{proof}
We can now see that $\mathcal{R}_D(A/A_0)$ is a division ring. To get an interpretable pseudofinite field, we need to define another ring. Let $\widetilde{\mathcal{R}}_{D}(A/A_0)$ be the ring of endomorphisms of $A/A_0$ generated by $D$ and the definable set $$\{(d-d')^{-1}:d, d'\in D, d-d'\neq\mathbf{0}\}$$ (the existence of $(d-d')^{-1}$ as automorphisms of $A/A_0$ is guaranteed by Lemma \ref{lem4.0}). 

By exactly the same proof, we can show that every non-zero element of $\widetilde{\mathcal{R}}_{D}(A/A_0)$  is an automorphism of $A/A_0$.

\begin{lemma}\label{lem5.0}
The division ring $\widetilde{\mathcal{R}}_D(A/A_0)$ is interpretable.
\end{lemma}
\begin{proof}
Pick some $[a]\neq[\mathsf{0}]$. For any $r\in\widetilde{\mathcal{R}}_D(A/A_0)$ with $r\neq\mathbf{0}$, consider the set $[a]^{Dr}$ which is the image of $[a]^D$ under $r$. Since $\mathsf{dim}([a]^D)=\mathsf{dim}(a^D)=1$ and $\mbox{ker}(r)$ is of dimension $0$ (as $r\neq\mathbf{0}$), we have that $[a]^{Dr}$ is of dimension 1. We claim that $$([a]^D-[a]^D)\cap ([a]^{Dr}-[a]^{Dr})\neq\{[\mathsf{0}]\}.$$ Indeed, if $([a]^D-[a]^D)\cap ([a]^{Dr}-[a]^{Dr})=\{[\mathsf{0}]\}$, then $[a]^{d_1}+[a]^{d_{2}r}=[a]^{d_3}+[a]^{d_4r}$ if and only if $[a]^{d_1}=[a]^{d_3}$ and $[a]^{d_2r}=[a]^{d_4r}$ for any $d_1,d_2,d_3,d_4\in D$. Hence any element in $[a]^D+[a]^{Dr}$ can be uniquely written as the sum. Therefore, $$\mathsf{dim}([a]^D+[a]^{Dr})=\mathsf{dim}([a]^D)+\mathsf{dim}([a]^{Dr})=2,$$ which contradicts the fact that $[a]^D+[a]^{Dr}$ is a subset of $A/A_0$ and $A/A_0$ is of dimension 1. Hence, there is some $d_1,d_2,d_3,d_4\in D$ such that $[a]^{d_1-d_2}=[a]^{(d_3-d_4)r}\neq [\mathsf{0}]$, i.e., $[a]^{(d_3-d_4)(d_3-d_4)^{-1}(d_1-d_2)}=[a]^{(d_3-d_4)r}$. Since $[a]\neq [\mathsf{0}]$ and $d_3-d_4$ is an automorphism, $[a]^{d_3-d_4}\neq [\mathsf{0}]$. Thus, $r=(d_3-d_4)^{-1}(d_1-d_2)$.

Therefore, $\widetilde{\mathcal{R}}_D(A/A_0)$ is a subset of $$E/\sim:=\{(d_3-d_4)^{-1}(d_1-d_2):d_1,d_2,d_3,d_4\in D, d_3-d_4\neq\mathbf{0}\}/\sim,$$ where $r\sim r'$ if $r$ and $r'$ induces the same endomorphism on $A/A_0$ for $r,r'\in E$.  On the other hand, $E/\sim$ is clearly a subset of $\widetilde{\mathcal{R}}_D(A/A_0)$. Since $E$ is definable, $\widetilde{\mathcal{R}}_D(A/A_0)$ is interpretable.
\end{proof}

Now we prove Theorem \ref{thm2D}.

\begin{proof}
By Lemma \ref{lem5.0}, $\widetilde{\mathcal{R}}_D(A/A_0)$ is an interpretable domain. Thus, there is some $J$ in the ultrafilter $\mathcal{U}$ such that $\widetilde{\mathcal{R}}_{D_i}(A_i/(A_0)_i)$ is also a finite domain in $(G_i,X_i)$ for any $i\in J$. Any finite domain is a field (Wedderburn's Little Theorem). Therefore, it is also true for all pseudofinite domain and we get $F:=\widetilde{\mathcal{R}}_D(A/A_0)$ is a field. It is an interpretable pseudofinite field.

Consider $D_0=\{d\in D: \forall a\in A, a^d\in a+A_0\}$. Take any $a\not\in A_0$, we know the set $[a]^D\subseteq A/A_0$ has dimension 1. Hence, $D/D_0$ has dimension at least 1.

By definition of $F=\widetilde{\mathcal{R}}_D(A/A_0)$ we know that $D/D_0$ embeds into $F^\times$. Hence $\mathsf{dim}(F)\geq 1$ and $D/D_0$ is commutative.

For any $[a]\neq [\mathsf{0}]$, let $[a]^{F}:=\{[a]^r:r\in F\}$. Define a map $i_a:F^+\to [a]^{F}$ by sending $r$ to $[a]^r$. It is clearly well-defined, surjective and is a group homomorphism. It is also injective. Indeed, if $[a]^r=[a]^{r'}$ for some $r,r'\in F$, then $[a]^{(r-r')}=[\mathsf{0}]$. Hence $r-r'=\mathbf{0}$, and we get $r=r'$. Therefore, $F^+$ is isomorphic to $[a]^{F}$. Note that $[a]^{F}$ is a definable subgroup of $A/A_0$. Moreover, it is of dimension 1, since $\mathsf{dim}(F)\geq 1$. We claim that $a^{F}=A/A_0$. If there is $[b]\in (A/A_0)\setminus [a]^{F}$, then $[b]^{F}$ is also isomorphic to $F^+$ and of dimension 1. As $[a]^{F}$ and $[b]^{F}$ are wide subgroups of $A$, we have $[a]^{F}\cap [b]^{F}$ is of dimension 1. In particular, there is $[c]\neq [\mathsf{0}]$. such that $[c]=[b]^{r_1}=[a]^{r_2}$ for some $r_1,r_2\neq\mathbf{0}$. Therefore, $[b]=[a]^{r_2r_1^{-1}}$ and $[b]\in [a]^{F}$, a contradiction.

Finally, we check that $\mathsf{dim}(D/D_0)=1$. By the proof before, we know that $D/D_0$ is of dimension at least 1. On the other hand, we also have $\mathsf{dim}(D/D_0)\leq \mathsf{dim}(F^\times)=\mathsf{dim}(F^+)=\mathsf{dim}(A)=1$. Hence, $\mathsf{dim}(D/D_0)=1$ as we have claimed.
\end{proof}

\begin{lemma}\label{lem field automorph}
Suppose $A$ is an abelian group of dimension 1 and $M$ is a group of automorphisms of $A$. Let $D\trianglelefteq M$ be a broad definable finite-by-abelian subgroup such that $A_0:=\{a\in A:\mathsf{dim}(a^D)=0\}$ is of dimension 0. Then $D$ satisfies the condition $(\clubsuit)$. Let $F:=\widetilde{\mathcal{R}}_D(A/A_0)$ be the interpretable pseudofinite field defined as in Theorem \ref{thm2D}. Then $M$ acts naturally by automorphisms on $F$ and $\mathsf{PStab}_M(F)/M_0$ embeds into $F^\times$ with $\mathsf{dim}(\mathsf{PStab}_M(F)/M_0)=1$, where $\mathsf{PStab}_M(F)$ is the point-wise stabilizer of $F$ and $$M_0:=\{m\in \mathsf{PStab}_M(F):\forall a\in A, ~a^m\in a+A_0\}.$$
\end{lemma}
\begin{proof}
Note that $A_0$ is definable by Lemma \ref{claim-0dimension}. And clearly, it is a $D$-invariant subgroup of $A$, so the induced action of $D$ on $A/A_0$ is well-defined. By the remark following Theorem \ref{thm2D}, we have that $D$ satisfies the condition $(\clubsuit)$.

Note that for any $a\in A$ and $m\in M$, if $\mathsf{dim}(a^D)=0$, then $\mathsf{dim}((a^m)^D)=\mathsf{dim}((a^D)^m)=0$. Therefore, $M$ also acts by automorphisms on $A/A_0$.

We define an action of $M$ on $F=\widetilde{\mathcal{R}}_D(A/A_0)$ by conjugation, i.e., for any $h\in M$ and $r\in F$, define $r^h:=h^{-1}rh$ (as the composition of automorphisms of $A/A_0$). We claim that $r^h\in F$ for any $r\in F$ and $h\in M$.

We prove by induction on the construction of $r\in F$:
\begin{enumerate}
\item
If $r=d\in D$, then $d^{h}=h^{-1}dh\in D$, as $D$ is normal in $M$.
\item
If $r=(d_1-d_2)^{-1}$ for some $d_1d_2^{-1}\not\in D_0$, then for any $[x],[y]\in A/A_0$, we have
\begin{align*}
 [x]^{r^h}=[y] &\text{~~ if and only if ~~} [x]^{h^{-1}(d_1-d_2)^{-1}h}=[y]\\ 
 &\text{~~ if and only if ~~} [x]=[y]^{h^{-1}(d_1-d_2)h} \\
 &\text{~~ if and only if ~~} [x]=[y]^{(d_1)^h-(d_2)^h}\\
 & \text{~~ if and only if ~~} [x]^{((d_1)^h-(d_2)^h)^{-1}}=[y].
\end{align*}

Thus, $r^h=((d_1)^h-(d_2)^h)^{-1}\in F$.
\item
If $r=r_1+r_2$, then $r^h=h(r_1+r_2)h^{-1}=(r_1)^h+(r_2)^h$. By induction hypothesis $(r_1)^h,(r_2)^h\in F$, hence $r^{h}\in F$.
\item
If $r=r_1r_2$, then $r^h=hr_1r_2h^{-1}=(r_1)^h(r_2)^h$. Again by induction hypothesis $(r_1)^h,(r_2)^h\in F$, hence $r^h\in F$.
\end{enumerate}

Clearly, for any $h\in M$ the map $(\cdot)^h$ is a field endomorphism, whence by pseudofiniteness, $(\cdot)^h$ is surjective, whence a field automorphism of $F$.

Consider the group $T:=\mathsf{PStab}_M(F)$. Let $T_0:=\{t\in T:\forall a\in A,~a^t\in a+A_0\}$. Note that $T_0$ is normal in $T$ as $T$ acts on $A_0$. Since $D/D_0$ is abelian and $D_0\subseteq T_0$, we have $DT_0/T_0\leq Z(T/T_0)$. For any $m_1,\ldots,m_n\in T$ and $a\not\in A_0$, we have $[a]^{C_{T/T_0}([m_1],\ldots,[m_n])}\supseteq [a^D]$, thus $\mathsf{dim}([a]^{C_{T/T_0}([m_1],\ldots,[m_n])})=1$. Therefore, we may apply Theorem \ref{thm2D} with $A, A_0$ and  $T$ and get an interpretable pseudofinite field $\bar{F}$ such that $A/A_0\simeq \bar{F}^+$, $T/T_0$ embeds into $\bar{F}^\times$ and $\mathsf{dim}(T/T_0)=1$. Note that $F\subseteq\bar{F}$ and $F^+\simeq A/A_0\simeq \bar{F}^+$, by pseudofiniteness $\bar{F}=F$.
\end{proof}

We now specify the case for $(G,X)\in\mathcal{S}$ with $\mathsf{dim}(G)=2$. Basically, we will apply Theorem \ref{thm2D} to get the interpretable field. However, we still need to find a definable normal abelian subgroup in $G$. This is the aim of the following two lemmas.

\begin{lemma}\label{lem1}
Let $(G,X)\in\mathcal{S}$ with $\mathsf{dim}(G)=2$. Then $G$ has no definable wide finite-by-abelian subgroup.
\end{lemma}
\begin{proof}
Suppose $G$ has such a subgroup $A$. By the $\widetilde{\mathfrak{M}_c}$-condition, we can take $D:=C_G(\bar{g})$ minimal up to finite index for some finite tuple $\bar{g}$ in $G$ such that $[A:A\cap D]<\infty$. 

We claim that $A\cap D\leq\widetilde{Z}(D)$. As $A$ is finite-by-abelian, we have $[A:C_A(a)]<\infty$ for any $a\in A\cap D$. Together with $[A:A\cap D]<\infty$, we get $[A:C_A(a)\cap D]< \infty$. Since $C_A(a)\cap D\leq C_D(a)$, also $[A:A\cap C_D(a)]<\infty$. By minimality of $D$ we have $[D:C_D(a)]<\infty$. Hence, $a\in \widetilde{Z}(D)$ and $A\cap D\leq \widetilde{Z}(D)$ as claimed. Since $A\cap D$ has finite index in $A$ and $A$ is wide, $\widetilde{Z}(D)$ is also wide in $G$.

By Lemma \ref{claim 2}, there is a definable wide normal subgroup $N\trianglelefteq G$ such that $N$ is commensurable with $\bigcap_{i\leq k}D^{g_i}$ for some $g_0,\ldots,g_k\in G$. By Lemma \ref{claim 4}, we have $\bigcap_{i\leq k}\widetilde{Z}(D)^{g_i}\leq \widetilde{Z}(\bigcap_{i\leq k}D^{g_i})$. Since $\widetilde{Z}(D)$ is wide, so is $\bigcap_{i\leq k}\widetilde{Z}(D)^{g_i}$, hence also $\widetilde{Z}(\bigcap_{i\leq k}D^{g_i})$. Since $N$ is commensurable with $\bigcap_{i\leq k}D^{g_i}$, we get $\mathsf{dim}(\widetilde{Z}(N))=\mathsf{dim}(\widetilde{Z}(\bigcap_{i\leq k}D^{g_i}))=2$ by Lemma \ref{claim 3}. Thus, $\widetilde{Z}(N)$ is a definable normal finite-by-abelian subgroup of $G$. Since $\widetilde{Z}(N)'$ is finite and normal in $G$, it is trivial by definably primitivity. Thus, $\widetilde{Z}(N)$ is a definable normal abelian subgroup of $G$. By Lemma \ref{claim 5.1}, $\mathsf{dim}(\widetilde{Z}(N))=1$, contradicting that $\mathsf{dim}(\widetilde{Z}(N))=2$.
\end{proof}

\begin{lemma}\label{lem2}
Let $(G,X)\in\mathcal{S}$ with $\mathsf{dim}(G)=2$. Assume that the definable sections of $G$ also satisfy the $\widetilde{\mathfrak{M}_c}$-condition. Then $G$ has a definable normal abelian subgroup $A$ of dimension $1$.
\end{lemma}
\begin{proof}
By Fact \ref{fact-frank-soluble}(1), $G$ has a broad definable finite-by-abelian subgroup $C$ whose normalizer is wide. We refer to the proof in \cite[Theorem 13]{Mcgroups}, from the construction of $C$ in the proof, there are two cases. The first case is that
$C$ is normal in $G$. Then $C$ is not wide by Lemma \ref{lem1}, so $\mathsf{dim}(C)=1$. Since $C'$ is definable normal in $G$ of dimension 0, it is trivial. Therefore, $A:=C$ is a definable normal abelian group of dimension 1.

The second case is that
$C:=\widetilde{Z}(D)$ where $D$ is commensurable with $E=C_G(\bar{b})$ for some $\bar{b}\in G^n$ and $\mathsf{dim}(D)\geq 1$. By the $\widetilde{\mathfrak{M}_c}$-condition and Schlichting's Theorem, there is a definable normal subgroup $H$ of $G$, such that $H$ is commensurable with $\bigcap_{i\leq k}E^{g_i}$, for some $g_0,\ldots,g_k\in G$. We may assume that $\mathsf{dim}(\widetilde{Z}(H))=\mathsf{dim}(\widetilde{Z}(\bigcap_{i\leq k}E^{g_i})=0$, for otherwise, we are in the previous case. Since $H$ is normal in $G$ and $\widetilde{Z}(H)$ is characteristic in $H$, $\widetilde{Z}(H)$ is a definable normal subgroup of $G$ of dimension 0. Hence $\widetilde{Z}(H)$ cannot act transitively on $X$ and is trivial by Lemma \ref{lem-trivialtansitive}. By Lemma \ref{claim 4} and Lemma \ref{claim 3}, we get $\bigcap_{i\leq k}\widetilde{Z}(E^{g_i})\leq \widetilde{Z}(\bigcap_{i\leq k}E^{g_i})$ and $\widetilde{Z}(\bigcap_{i\leq k}E^{g_i})$ is commensurable with $\widetilde{Z}(H)$. Hence $\bigcap_{i\leq k}\widetilde{Z}(E^{g_i})=\bigcap_{i\leq k}\widetilde{Z}(E)^{g_i}$ is finite.

As $D$ is commensurable with $E$, we have $\widetilde{Z}(D)$ is commensrable with $\widetilde{Z}(E)$. We may assume $[\widetilde{Z}(D):\widetilde{Z}(D)\cap \widetilde{Z}(E)]\leq \ell$ for some $\ell\in\mathbb{N}$. Then \begin{align*}
&\left[\bigcap_{i\leq k}\widetilde{Z}(D)^{g_i}:\left(\bigcap_{i\leq k}\widetilde{Z}(D)^{g_i}\right)\cap \left(\bigcap_{i\leq k}\widetilde{Z}(E)^{g_i}\right)\right]
\\
\leq &\prod_{j\leq k}\left[\bigcap_{i\leq k}\widetilde{Z}(D)^{g_i}:\left(\bigcap_{i\leq k}\widetilde{Z}(D)^{g_i}\right)\cap \widetilde{Z}(E)^{g_j}\right]\\
= &\prod_{j\leq k}\left[\bigcap_{i\leq k}\widetilde{Z}(D)^{g_i}:\left(\bigcap_{i\neq j}\widetilde{Z}(D)^{g_i}\right)\cap \left(\widetilde{Z}(D)^{g_j}\cap\widetilde{Z}(E)^{g_j}\right)\right]\\
= &\prod_{j\leq k}\left[\widetilde{Z}(D)^{g_j}:\widetilde{Z}(D)^{g_j}\cap\widetilde{Z}(E)^{g_j}\right]\leq\ell^{k+1}.
\end{align*}
As $\bigcap_{i\leq k}\widetilde{Z}(E)^{g_i}$ is finite, we get $\bigcap_{i\leq k}\widetilde{Z}(D)^{g_i}$ is also finite.

By assumption, $N_G(\widetilde{Z}(D))$ is wide, hence $\mathsf{dim}(N_G(\widetilde{Z}(D))/\widetilde{Z}(D))=1$. By Fact \ref{fact-frank-finitebyabelian}, there is a definable $B\leq N_G(\widetilde{Z}(D))$ such that $B/\widetilde{Z}(D)$ is broad finite-by-abelian. Hence, $B$ is wide in $G$. Clearly, $B^{g_i}/\widetilde{Z}(D)^{g_i}$ is also broad finite-by-abelian for any $g_i$. By Lemma \ref{claim 5}, the group $\bigcap_{i\leq k}B^{g_i}/\bigcap_{i\leq k}\widetilde{Z}(D)^{g_i}$ is finite-by-abelian. Since $\bigcap_{i\leq k}\widetilde{Z}(D)^{g_i}$ is finite, $\bigcap_{i\leq k}B^{g_i}$ is finite-by-abelian. However, $\bigcap_{i\leq k}B^{g_i}$ is definable and wide in $G$, contradicting Lemma \ref{lem1}.

\end{proof}

Now we can conclude the dimension-2 case.
\begin{theorem}\label{thm-dim2}
Let $(G,X)\in\mathcal{S}$ with $\mathsf{dim}(G)=2$. Suppose the definable sections of $G$ satisfy the $\widetilde{\mathfrak{M}_c}$-condition. Then $G=A\rtimes G_x$ and there is an interpretable pseudofinite field $F$ such that $A\simeq F^+$ and $D$ embeds into $F^\times$ for some wide definable subgroup $D\trianglelefteq G_x$. 

Moreover, $G_x$ induces a group of automorphisms on $F$.
\end{theorem}

\begin{proof}
By Lemma \ref{lem2}, $G$ has a definable normal abelian subgroup $A$. By Lemma \ref{claim 5.1} we have $G=A\rtimes G_x$ and $G_x$ acts on $A$ by conjugation, where $G_x$ is the point-stabilizer $\mathsf{Stab}_G(x)$. By Fact \ref{fact-frank-finitebyabelian}(2), $G_x$ has a definable finite-by-abelian normal subgroup $D$.  For any $a\in A$, if $\mathsf{dim}(a^D)=0$, then $\mathsf{dim}(C_D(a))=\mathsf{dim}(D)=1$. Since $A\times C_D(a)\subseteq C_G(a)$, we get $\mathsf{dim}(C_G(a))\geq \mathsf{dim}(A\rtimes C_D(a))=2=\mathsf{dim}(G)$. So $a=\mathsf{0}$ by Lemma \ref{lem3.0}. Therefore, $A_0:=\{a\in A:\mathsf{dim}(a^D)=0\}=\{\mathsf{0}\}$.  Applying Theorem \ref{thm2D} and Lemma \ref{lem field automorph} with $A_0=\{\mathsf{0}\}$ and $D_0=\{\mathsf{1}\}$, we get the desired result.
%a pseudofinite field $K=\widetilde{End}_D(A)$ and $G_x$ acts on $K$ by automorphisms. Moreover, $D\leq PS_{G_x}(K)$ embeds into $K^\times$. As $K$ is the ring of endomorphisms of $A$ generated by $D\cup\{(d-d')^{-1}:d,d'\in D,d\neq d'\}$, we have $PS_{G_x}(K)=C_{G_x}(D)$ . Let $H:=C_{G_{x}}(D)$. Then $H$ embeds into $K^\times$, whence is abelian. As $H=C_{G_x}(D)$ and $D\leq H$, we get $H$ is maximal abelian in $G_x$.
\end{proof}

If we add some extra condition on sets of dimension 0, we can also make the full stabilizer $G_{x}$ embeds into $F^\times$ as in Fact \ref{Elwes-main}.  

\begin{lemma}\label{lem-Btrivial}
Suppose an infinite field $F$ and a group $B$ of automorphisms of $F$ are interpretable in a theory with an additive integer-valued dimension $\mathsf{dim}$ such that $\mathsf{dim}(F)=1$. Then $B$ is either trivial or infinite.
\end{lemma}
\begin{proof}
If $B$ is finite, then any $\sigma\in B$ must have finite order. Thus, the fixed field  $fix(\sigma)$ is of finite index in $F$. As $1=\mathsf{dim}(F)=[F:fix(\sigma)]\cdot\mathsf{dim}(F)$, we get $fix(\sigma)=F$. Thus, $B$ is trivial. 
\end{proof} 
%For each element $e\in E_K$, the group generated by $e$, $E_e:=\langle e\rangle$, is finite of order $m$ for some $m\in\omega$. We follow the proof of \cite[Section 5.2, Claim 5]{Elwes} to show that $E_e$ is trivial. Let $H$ be the fixed field of $e$. There is $J\in \mathcal{U}$, such that for all $i\in J$, $(E_e)_i$ is a group of order $m$ of automorphisms of the finite field $K_i$, and $H_i$ is the fixed field of $e_i$. Hence, each $K_i$ is a definable $m$-dimensional vector space over $H_i$. Therefore, $\mathsf{dim}(K)=m\mathsf{dim}(H)$. By Theorem \ref{thm2D}, $\mathsf{dim}(K)=1$, we have $m=1$ and $K=H$. Thus, $(\cdot)^h:K\to K$ fixes the whole of $K$ for any $h\in G_{x}$.
\begin{corollary}\label{cor1}
Let $(G,X)\in\mathcal{S}$ with $\mathsf{dim}(G)=2$. Suppose the definable sections of $G$ satisfy $\widetilde{\mathfrak{M}_c}$-condition, and that the dimension-0 group $E_F:=G_{x}/\mathsf{PStab}_{G_{x}}(F)$ is finite. Then $G_{x}$ embeds into $F^\times$.
\end{corollary}
\begin{proof}
By the argument before, $(G,X)$ interprets a pseudofinite field $F$ of dimension $1$ and a group of field automorphisms $E_F:=G_{x}/\mathsf{PStab}_{G_{x}}(F)$.  By assumption, the group $E_F$ is finite, hence is trivial by Lemma \ref{lem-Btrivial}.
By Lemma \ref{lem field automorph}, $G_{x}=\mathsf{PStab}_{G_{x}}(F)$ embeds into $F^\times$. 
%Therefore, for any $h\in G_{x_0}$ and $d\in D$, we have $d^h=d$, whence $D\leq Z(G_{x_0})$. For any $n\in\mathbb{N}$, $d_1,\ldots,d_n\in G_{x_0}$ and $a\in A\setminus\{0\}$, we have $\mathsf{dim}(a^{C_{G_{x_0}}(d_1,\ldots,d_n)})\geq \mathsf{dim}(a^D)=1$. Applying Theorem \ref{thm2D} with $A_0=\emptyset$ and $G_{x_0}$, we get an interpretable pseudofinite field $\bar{K}:=\widetilde{\mathcal{R}}_{G_{x_0}}(A)$, where $\widetilde{\mathcal{R}}_{G_{x_0}}(A)$ is the ring of endomorphisms generated by $\{(g_1-g_2)^{-1}: g_1\neq g_2\in G_{x_0}\} \cup G_{x_0}$, such that $G_{x_0}$ embeds into $\bar{K}^\times$ and $(\bar{K})^+$ is isomorphic to $A$. Indeed, since $K^+\simeq A\simeq (\bar{K})^+$, $K^+\leq \bar{K}^+$ and $K,K^+$ are pseudofinite, we have $K=\bar{K}^+$.
\end{proof}

\section{Permutation Groups of Dimension $\geq 3$}\label{sec5}
This section deals with permutation groups in $\mathcal{S}$ of dimension greater or equal to $3$. The general strategy will be different from the previous sections. All the proofs before rely more on the $\widetilde{\mathfrak{M}_c}$-condition and properties of dimensions. From now on we will use pseudofiniteness to go directly to finite structures, and then use the well-established results of finite groups, such as CFSG (the classification of finite simple groups).

Moreover, as mentioned in the introduction, we need two extra assumptions: the $\widetilde{\mathfrak{M}_s}$-condition on $(G,X)$, and the (EX)-condition on $X$. 

While we need these two additional assumptions in the main theorem, we still make our statements as general as possible. 

The following lemma only assume pseudofiniteness and the $\widetilde{\mathfrak{M}_c}$-condition.
\begin{lemma}\label{lem-boundrank}
Let $G=\prod_{i\in I}G_i/\mathcal{U}$ be a pseudofinite $\widetilde{\mathfrak{M}_c}$- group. Then there is some $n<\omega$ and $J\in\mathcal{U}$ such that for all $i\in J$ we cannot find subgroups $D^i_0,\ldots,D^i_{n-1}$ of $G_i$ which are center-less and commute with each other.
%For any $n\in D$, there is $J\in\mathcal{U}$ such that $G_i$ has subgroups $D^i_1,\ldots,D^i_n$ which are center-less and commute with each other for all $i\in J$.
\end{lemma}
\begin{proof}
This is standard. Fix any $d\in \mathbb{N}$. Let $n=(d+1)\cdot m$ such that $2^m>d$. If the claim is not true, then for all $J\in \mathcal{U}$ there is $i\in J$ such that there are subgroups $D^i_0,\ldots, D^i_{n-1}$ in $G_i$ as claimed. Let $J_0:=\{i\in I: G_i$ has centerless subgroups $D^i_0,\ldots, D^i_{n-1}$ which commute with each other.$\}$ Then $J_0\in \mathcal{U}$, since otherwise the complement would be in the ultrafilter which contradicts our assumption.

For $i\in J_0$, choose $1\neq g_j^i\in D_j^i$ for each $j<n$, and put $h_k^i=\prod_{j<m}(g^i_{km+j})$ for $k\leq d$. Clearly, for each $i\in J_0$ and for any $1\leq k\leq d$ we have 
\begin{align*}
&~[C_{G_i}(h^i_0,\ldots,h^i_{k-1}):C_{G_i}(h^i_0,\ldots,h^i_k)]\\
&\geq [\prod_{j< m}D^i_{km+j}:C_{D^i_{km}}(g^i_{km}) C_{D^i_{km+1}}(g^i_{km+1})\cdots C_{D^i_{km+m-1}}(g^i_{km+m-1})]\\
&\geq \prod_{j<m}[D^i_{km+j}:C_{D^i_{km+j}}(g^i_{km+j})]
\geq 2^m>d.
\end{align*} 
%For any $D^i_j,D^i_k$ with $i\neq k$, as they are centerless, if we take $d^i_j\in D^i_j\setminus \{\mathsf{1}\}$ and $d^i_k\in D^i_k\setminus \{\mathsf{1}\}$, then $$[D^i_jD^i_k:C_{D^i_jD^i_k}(d^i_jd^i_k)]=[D^i_jD^i_k:C_{D^i_j}(d^i_j)D^i_k]\cdot[C_{D^i_j}(d^i_j)D^i_k:C_{D^i_j}(d^i_j)C_{D^i_k}(d^i_k)]\geq 2^2.$$ Similarly, if we take $d^i_1\in D^i_1\setminus \{\mathsf{1}\},\ldots, d^i_m\in D^i_m\setminus \{\mathsf{1}\}$, and let $g:=\prod_{1\leq j\leq m}d^i_j$, then $$[\prod_{1\leq j\leq m}D^i_j:C_{\prod_{1\leq j\leq m}D^i_j}(g)]\geq 2^m>d.$$ As $n=(d+1)\cdot m$, for any $0\leq k\leq d$, we can find $g_k\in \prod_{1\leq j\leq m}D^i_{km+j}$ such that $$[C_{G_i}(g_0,\ldots,g_{k-1}):C_{G_i}(g_0,\ldots,g_k)]\geq [\prod_{1\leq j\leq m}D^i_{km+j}:C_{\prod_{1\leq j\leq m}D^i_{km+j}}(g_k)]>d.$$ As this holds any for $i\in J$, it will also hold in $G$. 
Hence, $G$ does not satisfy the $\widetilde{\mathfrak{M}_c}$-condition, a contradiction.
\end{proof}

Suppose $G=\prod_{i\in I}G_i/\mathcal{U}$. Let $H_i$ be a non-trivial minimal normal subgroup in $G_i$ for $i\in I$. Then $H_i$ is a direct product of isomorphic simple groups. Suppose $H_i=T_i\odot T_i^{g_{i_1}}\odot\cdots\odot T_i^{g_{i_{n_i}}}$ with $g_{i_1},\ldots,g_{i_{n_i}}\in G_i$ and $T_i$ simple. If $H_i$ is not abelian, then neither is $T_i$. Let $H:=\prod_{i\in I}H_i/\mathcal{U}$ and $T=\prod_{i\in I}T_{i}/\mathcal{U}$. 

\begin{lemma}\label{lemT}
Let $(G,X)\in\mathcal{S}$. In particular, $G$ is a pseudofinite $\widetilde{\mathfrak{M}_c}$-group. Let $H$ be defined as above. If $H$ is not abelian, then $T$ is infinite and there is $m\in\mathbb{N}$ such that $H=T\odot T^{g_1}\odot\cdots \odot T^{g_m}$ for some $g_1,\ldots,g_m\in G$.

Moreover, $T$ and $H$ are definable, and $T$ is a simple pseudofinite group.
\end{lemma}
\begin{proof}
By Lemma \ref{lem-boundrank}, there is $m\in\mathbb{N}$ and $J\in\mathcal{U}$ such that  $H_i$ is $m+1$-fold product of conjugates of $T_i$ for all $i\in J$. Hence, $H=T\odot T^{g_1}\odot\cdots \odot T^{g_m}$ for some $g_1\ldots,g_m\in G$. We claim that $T$ is infinite. Otherwise, if $T$ is finite, then $H$ is finite, hence definable. Since $H$ is non-trivial, it acts transitively on $X$. Hence, $\mathsf{dim}(X)\leq\mathsf{dim}(H)=0$, a contradiction.

For each $i\in I$, since $T_i$ is non-abelian, we may assume it is either an alternating group $\textsf{Alt}_{n_i}$ or a classical group of Lie type of rank $n_i$ over some field $\mathbb{F}_{q_i}$, denoted as $\textsf{cl}_{n_i}(q_i)$. We claim that $n_i$ is bounded. If not, then for any $n$, for all large enough $n_i$, the group $\textsf{Alt}_{n_i}$ will contain at least $n$ commuting copies of $\textsf{Alt}_5$, and $\textsf{cl}_{n_i}(q_i)$ will contain at least $n$ commuting copies of $\text{PSL}_2(\mathbb{F}_{p_i})$, where $p_i$ is the characteristic of $\mathbb{F}_{q_i}$. Both cases contradict Lemma \ref{lem-boundrank}. Thus, we may assume $\{T_i:i\in I\}$ are classical groups of Lie type of bounded Lie rank.

By \cite{wilson1995simple}, $T$ is a simple pseudofinite group. Hence, the theory of $T$ in the language of pure group is supersimple of finite $SU$-rank by \cite{ryten2007model}. As $T$ is infinite nonabelian simple, there is some $x\in T$ such that the set $x^T$ is infinite. By the Indecomposability Theorem (\ref{fact-indecomp}), there is some infinite definable group $D\leq x^T\cdots x^T$ which is normal in $T$, where $x^T\cdots x^T$ is a $k$-fold product for some $k\in\mathbb{N}$. Denote the $k$-fold product of $X$ as $X\cdot (k)\cdot X$. Since $T$ is simple, $D=T$. Therefore, $x^T\cdot (k)\cdot x^T=T$. As $H$ is normal and $x\in H$, we have \begin{align*}
H&\supseteq (x^G\cdot (k)\cdot x^G)\odot (x^G\cdot (k)\cdot x^G)^{g_1}\odot\cdots\odot(x^G\cdot (k)\cdot x^G)^{g_m}\\
&\supseteq T\odot T^{g_1}\odot\cdots\odot T^{g_m}=H.
\end{align*}
Consequently, $H$ is definable. Moreover, since $x^H\cdot (k)\cdot x^H=x^T\cdot(k)\cdot x^T=T$, we also get $T$ definable. 
\end{proof}

\begin{lemma}\label{lemGx}
Let $(G,X)\in\mathcal{S}$. Suppose $G$ satisfies the $\widetilde{\mathfrak{M}_s}$-condition. Let $D$ be a normal definable subgroup of $G$. Suppose $\mathsf{dim}(D)=n$. Then there are $x_1,\ldots,x_n\in X$ such that for all $1\leq i\leq n$ we have $$\mathsf{dim}(\mathsf{PStab}_{D}(x_1,\ldots,x_i))=n-i.$$
Moreover, there are $x_1,\ldots,x_t\in X$ such that $\mathsf{PStab}_D(x_1,\ldots,x_t)=\{\mathsf{1}\}$.
\end{lemma}
\begin{proof}
To prove the first part of the statement, it suffices to show that there are $x_1,\ldots,x_n\in X$ such that $\mathsf{dim}(\mathsf{PStab}_D(x_1,\cdots,x_n))=0$.
Since $(G,X)$ satisfies $\widetilde{\mathfrak{M}_s}$-condition, so does $(D,X)$. By the $\widetilde{\mathfrak{M}_s}$-condition, there are $x_1,\ldots,x_m\in X$ and $d\in \mathbb{N}$ such that $$[\mathsf{PStab}_D(x_1,\ldots,x_m):\mathsf{PStab}_D(x_1,\ldots,x_m,x)]\leq d,$$ for any $x\in X$. As $D$ is normal in $G$, we get $\{(\mathsf{PStab}_D(x_1,\ldots,x_m))^g:g\in G\}$ is a uniformly commensurable family of definable subgroups. By Schlichting's Theorem, there is definable $D_0\trianglelefteq G$ such that $D_0$ is commensurable with $\mathsf{PStab}_D(x_1,\ldots,x_m)$. 
By Lemma \ref{lem-trivialtansitive}, either $x^{D_0}=X$ or $D_0$ is trivial. If $x^{D_0}=X$, then $$\mathsf{dim}(x^{\mathsf{PStab}_D(x_1,\ldots,x_m)})=\mathsf{dim}(x^{D_0})=1.$$ By the Orbit-Stabilizer Theorem  $$|x^{\mathsf{PStab}_D(x_1,\ldots,x_m)}|=[\mathsf{PStab}_D(x_1,\ldots,x_m):\mathsf{PStab}_D(x_1,\ldots,x_m,x)]\leq d,$$ a contradiction. Therefore, $D_0$ is trivial. As $\mathsf{PStab}_D(x_1,\ldots,x_m)$ is commensurable with $D_0$, we deduce $\mathsf{PStab}_D(x_1,\ldots,x_m)$ is finite. So we only need finitely many more points, say $x_{m+1},\ldots,x_t\in X$, to distinguish $\mathsf{1}$ from other elements in $\mathsf{PStab}_D(x_1,\ldots,x_m)$. Therefore, $\mathsf{PStab}_D(x_1,\ldots,x_t)=\{\mathsf{1}\}$.

To finish the proof we show that there is a subsequence $x_{i_1},\ldots,x_{i_n}$ of $x_1,\ldots,x_m$ with $\mathsf{dim}(\mathsf{PStab}_D(x_{i_1},\ldots,x_{i_n}))=0$. Consider the dimensions of the following sequence $$\mathsf{PStab}_D(x_1),\mathsf{PStab}_D(x_1,x_2),\ldots,\mathsf{PStab}_D(x_1,\ldots,x_m).$$ By the Orbit-Stabilizer Theorem, the dimension can drop at most 1 in each step. Hence, $m\geq n$. Take $n$ elements, say $x_{i_1},\ldots,x_{i_n}$ with $i_1<i_2<\cdots<i_n$, such that each of the corresponding dimension drops. By our choice, $$1\geq\mathsf{dim}((x_{i_j})^{\mathsf{PStab}_D(x_{i_1},\ldots,x_{i_{j-1}})})\geq\mathsf{dim}((x_{i_j})^{\mathsf{PStab}_D(x_1,x_2,\ldots,x_{i_j-1})})=1,$$ for each $1\leq j\leq n$. Therefore, $\mathsf{dim}(\mathsf{PStab}_D(x_{i_1},\ldots,x_{i_n}))=0$.\qedhere

%For the second claim of the lemma, by the $\widetilde{\mathfrak{M}_s}$-condition, there are $k\in\mathbb{N}$ and $x_1,\ldots,x_m\in X$ such that $[\mathsf{PStab}_G(x_1,\ldots,x_m):\mathsf{PStab}_G(x_1,\ldots,x_m,x)]\leq k$ for all $x\in X$. By Schlichting's Theorem, there is a definable normal subgroup $H\trianglelefteq G$ such that $H$ is commensurable with $\mathsf{PStab}_G(x_1,\ldots,x_m)$. By the first part of this lemma, we know that $\mathsf{dim}(\mathsf{PStab}_G(x_1,\ldots,x_m))=0$. As $G$ acts definably primitively on $X$, either $H$ is trivial or $\mathsf{dim}(H)=1$.  Since $H$ is commensurable with $\mathsf{PStab}_G(x_1,\ldots,x_m)$, the latter cannot happen. Therefore, $\mathsf{PStab}_G(x_1,\ldots,x_m)$ is finite. 
\end{proof}

\begin{lemma}\label{lem-xGy}
Suppose $(G,X)\in\mathcal{S}$ satisfies the $\widetilde{\mathfrak{M}_s}$-condition. Let $D$ be a non-trivial normal definable subgroup of $G$. For any $x\in X$, define $L_x:=\{y\in X:\mathsf{dim}(x^{D_y})=0\}$. Then $L_x$ is uniformly definable with respect to $x$.
\end{lemma}
\begin{proof}
First note that since $D$ is a definable subgroup of $G$, we have $(D,X)$ also satisfies $\widetilde{\mathfrak{M}_s}$-condition. Assume $\mathsf{dim}(D)=n$. Note that since $D$ is non-trivial, definable and normal, it acts transitively on $X$. Thus, $\mathsf{dim}(\mathsf{Stab}_D(x))=n-1$ for any $x\in X$. By the $\widetilde{\mathfrak{M}_s}$-condition, there are $x_1,\ldots,x_k\in X$ and $d\in\mathbb{N}$ such that $\mathsf{dim}(\mathsf{PStab}_D(x_1,\ldots,x_k))=n-1$ and for any $y\in X$, we have either $\mathsf{dim}(\mathsf{PStab}_D(x_1,\ldots,x_k,y))=n-2$ or $$[\mathsf{PStab}_D(x_1,\ldots,x_k):\mathsf{PStab}_D(x_1,\ldots,x_k,y)]\leq d.$$ As $\mathsf{dim}(D_{x_1})=\mathsf{dim}(\mathsf{PStab}_D(x_1,\ldots,x_k))=n-1$, we get $\mathsf{dim}(z^{D_{x_1}})=0$ if and only if $[\mathsf{PStab}_D(x_1,\ldots,x_k):\mathsf{PStab}_D(x_1,\ldots,x_k,z)]\leq d$ for any $z\in X$.

For any $y\in X$, let $g\in D$ be such that $(x_1)^{g}=y$. Then $y\in L_x$ if and only if $\mathsf{dim}(x^{D_{(x_1)^g}})=0$ if and only if $\mathsf{dim}((x^{g^{-1}})^{D_{x_1}})=0$ if and only if there is $g\in D$ such that $(x_1)^g=y$ and $$[\mathsf{PStab}_D(x_1,\ldots,x_k):\mathsf{PStab}_D(x_1,\ldots,x_k,x^{g^{-1}})]\leq d.\qedhere$$
\end{proof}

%In the following, we add the (EX)-condition.

\begin{theorem}\label{thm-normalabelian}
Suppose $(G,X)\in\mathcal{S}$ with $\mathsf{dim}(G)\geq 3$ satisfies the $\widetilde{\mathfrak{M}_s}$-condition and $X$ satisfies the (EX)-condition. Then $G$ does not contain any nontrivial abelian normal subgroup.
\end{theorem}
\begin{proof}
The theorem follows from the claims below. 
\begin{claim}
If $G$ has a nontrivial normal abelian subgroup $H$, then $G$ has a definable nontrivial normal abelian subgroup $A$.
\end{claim}
\begin{proof}
If $G$ has a non-trivial normal abelian subgroup, then $G$ has a definable finite-by-abelian subgroup $A$, which is normal in $G$ and contains $H$, by \cite[Proposition 4.17]{Hempel16}. Since $A'$ is definable and of dimension 0, by definable primitivity, $A'$ is trivial, hence $A$ is abelian. Since $A$ contains $H$, we get $A$ is nontrivial. 
\end{proof}

Suppose the conclusion of Theorem \ref{thm-normalabelian} fails, then $G$ has a nontrivial definable normal abelian subgroup $A$. By Lemma \ref{claim 5.1}, $G=A\rtimes G_x$ where $G_x:=\mathsf{Stab}_G(x)$ for some $x\in X$. We identify $A$ with $X$. Then $G_x$ acts on $A$ by conjugation, while $A$ acts on itself by addition. Our aim is to derive a contradiction.

\begin{claim}\label{lem primitivity}
Suppose $G\in\mathcal{S}$ and $\mathsf{dim}(G)\geq 2$. Assume $G=A\rtimes G_x$. Let $C\trianglelefteq G_x$ with $C$ definable and $\mathsf{dim}(C)\geq 1$. Then $A\rtimes C$ also acts definably primitively on $X$.  
\end{claim}
\begin{proof}
We may assume that $A\rtimes G_x=\prod_{i\in I}A_i\rtimes (G_x)_i/\mathcal{U}$, the formula defining $C$ also defines $C_i\trianglelefteq (G_x)_i$ for each $i\in I$. Let $W_i\leq A_i$ be a nontrivial $C_i$-irreducible subgroup, that is a minimal nontrivial $C_i$-invariant subgroup. Consider $W:=\prod_{i\in I}W_i/\mathcal{U}$. Then $W$ is nontrivial and $C$-invariant. If there is $V:=\prod_{i\in I}V_i/\mathcal{U}$ with each $V_i\neq W_i$ non-trivial and $C_i$-irreducible, then $W\cap V=\emptyset$. Take $a\in W\setminus\{\mathsf{0}\}$ and $b\in V\setminus\{\mathsf{0}\}$. Note that $A\rtimes C\trianglelefteq G$ and $\mathsf{dim}(A\rtimes C)\geq 2$. By Lemma \ref{lem3.0}, we have $C_{A\rtimes C}(a)$ and $C_{A\rtimes C}(b)$ are not wide in $A\rtimes C$. Therefore, $\mathsf{dim}(a^C)=\mathsf{dim}(b^C)=1$. Moreover, we have $(a^C-a^C)\cap(b^C-b^C)\subseteq W\cap V=\emptyset$. Hence, $\mathsf{dim}(a^C+b^C)=\mathsf{dim}(a^C)+\mathsf{dim}(b^C)=2$, contradiction. Hence, we may assume that there is only one nontrivial $C_i$-irreducible subgroup in any $A_i$. 

Let $H$ be any non-trivial definable $C$-invariant subgroup of $A$. Then each $H_i$ is non-trivial and $C$-invariant. Thus, $W_i\subseteq H_i$ and we get $W\subseteq H$. Since $C$ is normal in $G_x$, $H^g$ is also $C$-invariant for any $g\in G_x$. By the same argument, $W\subseteq H^g$. Therefore, $W\subseteq \bigcap_{g\in _{G_x}}H^g$. The group $M:=\bigcap_{g\in _{G_x}}H^g\leq A$ is non-trivial, definable and $G_x$ invariant. As $M\leq A$ is $G_x$ invariant and $G=A\rtimes G_x$, we have $M$ is normal in $G$. Since $M$ is nontrivial, it must act transitively on $X$ by Lemma \ref{lem-trivialtansitive}. As $A$ acts on $X$ regularly by Lemma \ref{claim 5.1}, we deduce $M=H=A$. Therefore, $A$ is the minimal non-trivial definable $C$-invariant subgroup of $A$.

Clearly, $\mathsf{Stab}_{A\rtimes C}(x)= C$. Suppose there is a definable group $C\leq D\leq A\rtimes C$, then $D\cap A\leq A$. Moreover, as $(D\cap A)^C\leq D^C\cap A^C=D\cap A$, we have $(D\cap A)^C=D\cap A$. As $A$ is the minimal non-trivial definable $C$-invariant subgroup of $A$, we conclude either $D\cap A=A$ or $D\cap A=\{\mathsf{0}\}$. Therefore, either $D=C$ or $D=A\rtimes C$.
\end{proof}

By Lemma \ref{lemGx}, there is $\bar{x}=(x_1,\ldots,x_{n-2})$ such that $\mathsf{dim}(\mathsf{PStab}_G(\bar{x}))=2$. We may assume $\mathsf{PStab}_G(\bar{x})\subseteq G_x$ and we write $\mathsf{PStab}_G(\bar{x})$ as $G_{\bar{x}}$. By Fact \ref{fact-frank-soluble}(1), $G_{\bar{x}}$ has a broad definable finite-by-abelian subgroup $D$ such that $N_{G_{\bar{x}}}(D)$ has dimension 2.

Consider the group $A^D_0:=\{a\in A:\mathsf{dim}(a^D)=0\}$. The dimension of $A^D_0$ is either 0 or 1. We will show that neither of them holds.

\begin{claim}\label{lemfixer}
The dimension of $A^D_0$ is not 1.
\end{claim}
\begin{proof}
Suppose $\mathsf{dim}(A^D_0)=1$. By Lemma \ref{claim-0dimension}, there are $d\in\mathbb{N}$ and a definable group $T\leq D$ such that $A^D_0=\{a\in A:[T:C_T(a)]\leq d\}$ and $\mathsf{dim}(T)=\mathsf{dim}(D)$. Therefore $A^D_0\leq \widetilde{C}_G(T)$. Since $A$ is in definable bijection with $X$, by the (EX)-condition, $A^D_0$ has finite index in $A$. Hence, $A\lesssim \widetilde{C}_G(T)$. By Fact \ref{fact-centralizer}, $T\lesssim \widetilde{C}_G(A)$.

Let $M:=\widetilde{C}_G(A)\cap G_x$. Then $\mathsf{dim}(M)\geq\mathsf{dim}(T)\geq 1$. Note that $\widetilde{C}_G(A)$ is normal in $G$, hence, $M$ is normal in $G_x$. By Lemma \ref{lem primitivity}, $A\rtimes M=\widetilde{C}_G(A)$ also acts definably primitively on $X$.

As $\widetilde{C}_G(A)\lesssim \widetilde{C}_G(A)$, we have $A\lesssim \widetilde{C}_G(\widetilde{C}_G(A))$ by Fact \ref{fact-centralizer}. Thus, there is $\mathsf{0}\neq a\in A$ such that $[\widetilde{C}_G(A):C_{\widetilde{C}_G(A)}(a)]<\infty$, which means $C_{\widetilde{C}_G(A)}(a)$ is wide in $\widetilde{C}_G(A)$, contradicting Lemma \ref{lem3.0}.
\end{proof}

\begin{claim}The dimension of $A^D_0$ is not 0.
\end{claim}
\begin{proof}
Let $M:=N_{G_{\bar{x}}}(D)$. As the normalizer of $D$ is wide in $G_{\bar{x}}$, we have $\mathsf{dim}(M)=2$. Suppose $\mathsf{dim}(A^D_0)=0$. We can apply Theorem \ref{thm2D} and Lemma \ref{lem field automorph} to get an interpretable pseudofinite field $F$ such that $A/A^D_0\simeq F^+$ and $M$ extends to a group of automorphisms of $F$. Consider the point-wise stabilizer $\mathsf{PStab}_{M}(F)$. Let $$M_0:=\{m\in \mathsf{PStab}_{M}(F):\forall a\in A,~a^m\in a+A^D_0\}.$$ By Lemma \ref{lem field automorph}, $\mathsf{dim}(\mathsf{PStab}_{M}(F)/M_0)=1$. By the second part of Lemma \ref{lemGx}, the value of $m\in M_0$ is determined by its value on some $a_1,\ldots,a_t\in A$. Hence, $$\mathsf{dim}(M_0)\leq t\mathsf{dim}(A^D_0)=0.$$ Thus, $\mathsf{dim}(\mathsf{PStab}_{M}(F))=1$.

Therefore, $T:=M/\mathsf{PStab}_{M}(F)$ is a group of automorphisms of $F$ such that the action is faithful and $\mathsf{dim}(T)=\mathsf{dim}(M)-\mathsf{dim}(\mathsf{PStab}_{M}(F))=2-1=1$.

Consider $F_0^T:=\{k\in F:\mathsf{dim}(k^T)=0\}$. By the fact that $T$ is a group of automorphisms of $F$, we can check easily that $F_0^T$ is a subfield of $F$. Note that $F_0^T$ is definable (apply Lemma \ref{claim-0dimension} to the group $(F^+\rtimes T)$). We claim that either $F_0^T=F$ or $\mathsf{dim}(F_0^T)=0$. Indeed, if $\mathsf{dim}(F_0^T)=1$, then $$1=\mathsf{dim}(F)=[F:F_0^T]\cdot\mathsf{dim}(F_0^T)=[F:F_0^T],$$ and we get $F=F_0^T$.

If $F_0^T=F$, then by the $\widetilde{\mathfrak{M}_c}$-condition of the interpretable group $F^+\rtimes T$, there are $k_0,\cdots,k_t\in F$ and $n\in\mathbb{N}$ such that if we define $H:=C_T(k_0,\cdots,k_t)$, then for all $k\in F$ we have $[H:C_H(k)]\leq n$, that is $|k^H|\leq n$. Consider the group $F^+\rtimes H$. From the above argument we know that $F^+\lesssim \widetilde{C}_{F^+\rtimes H}(H)$. By Fact \ref{fact-centralizer}, we have $H\lesssim \widetilde{C}_{F^+\rtimes H}(F^+)$. Therefore, there is $h\neq id$ such that $[F^+:C_{F^+}(h)]<\infty$. Since $C_{F^+}(h)$ is a definable subfield of $F$ and $\mathsf{dim}(F)=1$, we have $C_{F^+}(h)=F^+$, contradicting $h\neq id$.

Thus $F_0^T$ is of dimension $0$. Let $Y:=F\setminus F_0^T$. Clearly, there is $J\in \mathcal{U}$ such that $|Y_i|\geq |F_i|/2$ for all $i\in J$. If $F_i=\mathbb{F}_{p_i^{n_i}}$, then $|T_i|\leq n_i$. Therefore, there are infinitely many $T$-orbits on $Y$ and each of them has dimension 1. Note that $X$ is in definable bijection with $F^+$, contradicting the (EX)-condition.
\end{proof}

This finishes the proof of Theorem \ref{thm-normalabelian}.
\end{proof}

With all the assumptions above we conclude that $H=T\times T^{g_1}\times\cdots\times T^{g_m}$, with $T$ definable and simple non-abelian. The following lemmas show that $T$ is normal in $G$, hence $H=T$.

The following four lemmas all assume that $(G,X)\in\mathcal{S}$ satisfies the $\widetilde{\mathfrak{M}_s}$-condition and the (EX)-condition. 

\begin{lemma}\label{lem-finiteorbit}
Let $D$ be a non-trivial definable normal subgroup of $G$. Suppose $\mathsf{dim}(D)\geq 2$. Then for any $x\in X$, the group $D_x:=\mathsf{Stab}_D(x)$ has only finitely many orbits on $X$.
\end{lemma}

\begin{proof}
Note that $D$ is definable normal and non-trivial. It acts transitively on $X$. Therefore, $\mathsf{dim}(D)\geq \mathsf{dim}(x^D)=1$ and for any $x\in X$, $$\mathsf{dim}(D_x)=\mathsf{dim}(D)-\mathsf{dim}(x^D)=\mathsf{dim}(D)-1\geq 1.$$

Define a relation $\sim$ on $X$ as: $x\sim y$ if $\mathsf{dim}(x^{D_y})=0$. Clearly, $\sim$ is reflexive. It is symmetric. If $\mathsf{dim}(x^{D_y})=0$, then $\mathsf{dim}(D_y/D_{yx})=0$. Therefore, $\mathsf{dim}(D_{yx})=\mathsf{dim}(D_y)=\mathsf{dim}(D_x)$, and $y^{D_x}$ has dimension 0. It is also transitive. If both $x^{D_y}$ and $y^{D_z}$ have dimension 0, then $\mathsf{dim}(D_x)=\mathsf{dim}(D_{xy})=\mathsf{dim}(D_y)=\mathsf{dim}(D_{yz})$. That is, both $D_{xy}$ and $D_{yz}$ are wide in $D_y$. Therefore, $D_{xyz}=D_{xy}\cap D_{yz}$ is also wide in $D_y$. Hence $\mathsf
{dim}(D_{xyz})=\mathsf{dim}(D_y)=\mathsf{dim}(D_z)$. We get $\mathsf{dim}(x^{D_z})=\mathsf{dim}(D_z/D_{xz})\leq \mathsf{dim}(D_z/D_{xyz})=0$. 

Moreover, $\sim$ is $G$-invariant and definable. It is definable by Lemma \ref{lem-xGy}. For $G$-invariance, if $x\sim y$, then for any $g\in G$, we have $(x^g)^{D_{y^g}}=(x^g)^{(D_y)^g}=(x^{D_y})^g$. Thus, $\mathsf{dim}((x^g)^{D_{y^g}})=\mathsf{dim}(x^{D_y})=0$. Consequently, $x^g\sim y^g$. 

By definable primitivity, $\sim$ is either trivial or the universal congruence. By Lemma \ref{lemGx}, there is $y\in X$ such that $\mathsf{dim}(\mathsf{PStab}_D(x,y))=\mathsf{dim}(D_x)-1$. Thus, $\sim$ is not the universal congruence. Therefore, every $D_x$ orbit on $X\setminus\{x\}$ has dimension 1. By the (EX)-condition, there can be only finitely-many such orbits.
\end{proof}

\begin{lemma}\label{lem-E=H_x}
Let  $D$ be a normal definable subgroup of $G$ with $\mathsf{dim}(D)\geq 2$. Suppose there is a definable subgroup $E$ such that $\mathsf{Stab}_D(x)\leq E\leq D$ and $\mathsf{dim}(E)=\mathsf{dim}(\mathsf{Stab}_D(x))$. Then $E=\mathsf{Stab}_D(x)$.
\end{lemma}
\begin{proof}
Let $D_x:=\mathsf{Stab}_D(x)$. 
As $\mathsf{dim}(E)=\mathsf{dim}(D_x)$, we have $\mathsf{dim}((D_x)^m\cap D_x)=\mathsf{dim}(D_x)$ for any $m\in E$. 
Note that $\mathsf{dim}((D_x)^m\cap D_x)=\mathsf{dim}(D_x)$ if and only if  $\mathsf{dim}(D_{x^m}\cap D_x)=\mathsf{dim}(D_x)$ if and only if $\mathsf{dim}(x^{D_{x^m}})=0$ if and only if $x\sim x^m$, where $\sim$ is defined as in Lemma \ref{lem-finiteorbit}. By the same lemma, $x\sim y$ if only if $x=y$. Therefore, $x^m=x$ and $m\in D_x$. We conclude that $E=D_x$. 
\end{proof}

 \begin{lemma}\label{lem-subgroupindex}
If $D$ is a definable normal subgroup of $G$ of finite index and $\mathsf{dim}(D)\geq 2$, then $D$ also acts definably primitively on $X$.
 \end{lemma}
 \begin{proof}
Let $M$ be a definable subgroup of $D$ such that $D_x\leq M\leq D$, where $D_x:=\mathsf{Stab}_D(x)$. Then either $\mathsf{dim}(M)=\mathsf{dim}(D_x)=n-1$ or $\mathsf{dim}(M)=\mathsf{dim}(G)$.

%, then the set $x^M$ will have dimension 1. As $D_x\leq M\leq D$ and $D$ acts transitively on $X$, we get $D$ permutes the orbits of $M$ in $X$. 
If $\mathsf{dim}(M)=\mathsf{dim}(D)=\mathsf{dim}(G)$, then $$\mathsf{dim}(x^M)=\mathsf{dim}(M/M_x)=\mathsf{dim}(M/M\cap D_x)\geq \mathsf{dim}(D/D_x)=1.$$ Consider the right coset space of $M$ in $D$. Assume $D=\bigcup_{i\in I}Md_i$ with $Md_i\neq Md_j$ for $i\neq j$. Let $\mathcal{E}:=\{x^{Md_i}:i\in I\}$. We claim that $x^{Md_i}\cap x^{Md_j}=\emptyset$ for any $i\neq j$. Suppose $x^{Md_i}\cap x^{Md_j}\neq\emptyset$, then there are $m_i,m_j\in M$ with $x^{m_id_i}=x^{m_jd_j}$. Therefore, $m_id_i(d_j)^{-1}(m_j)^{-1}\in D_x$. As $D_x\leq M$, we get $d_i(d_j)^{-1}\in M$, hence $i=j$. Note that $\mathsf{dim}(x^{Md_i})=\mathsf{dim}(x^M)=1$ for all $i\in I$. By the (EX)-condition, $I$ must be finite. Consequently, $M$ has finite index in $D$, hence $[G:M]<\infty$. By Poincar\'{e}'s Theorem, $M$ contains a definable normal subgroup $S$ of $G$ which also has finite index in $G$. Therefore, $x^S=X$ and $x^M\supseteq x^S=X$. For any $d\in D$, there is $m\in M$ such that $x^d=x^m$. Thus, $dm^{-1}\in D_x\leq M$ and $d\in M$. Therefore, $D=M$.
%there are only finitely many such orbits $$x^M,(x^{d_1})^M,\ldots,(x^{d_{\ell}})^M.$$ Let $d_0:=\mathsf{1}$. For any $d\in D$, there are $0\leq i\leq \ell$ and $m\in M$ such that $x^d=x^{d_im}$. Hence, $d_imd^{-1}\in D_x\leq M$, which entails $d\in d_iM$. Therefore, $D=\bigcup_{i\leq\ell}d_i M$. 

Suppose $\mathsf{dim}(M)=\mathsf{dim}(D_x)$, then by Lemma \ref{lem-E=H_x}, we get $M=D_x$.  Therefore, $D$ acts definably primitively on $X$.
\end{proof}
 
\begin{lemma}\label{lem-C(H)}
Let $H=T\times T^{g_1}\times\cdots\times T^{g_m}$ be as above. Then $H=T$ and $C_G(H)$ is trivial. In fact, $H=\prod_{i\in I}soc(G_i)/\mathcal{U}$ where $soc(G_i)$ is the socle of $G_i$.
\end{lemma}
\begin{proof}
Consider $G_T:=\{g\in G:T^g=T\}$. As $\{T,T^{g_1},\ldots, T^{g_m}\}$ is permuted by $G$, the index of $G_T$ in $G$ is finite. By Poincar\'{e}'s Theorem, there is a definable normal subgroup $G_0:=\bigcap_{g\in G}(G_T)^g$, which also has finite index in $G$. By definition, $H\leq G_0$. By Lemma \ref{lem-subgroupindex}, $G_0$ also acts definably primitively on $X$. 

Note that $T$ is normal in $G_0$. Consider $S:=C_{G_0}(T)$. It is definable and normal in $G_0$. If $S$ is non-trivial, then $T$ and $S$ centralize each other and both act transitively on $X$. Fix $x\in X$. For any $h\in T$, we have $\mathsf{Stab}_S(x^h)=(\mathsf{Stab}_S(x))^h=\mathsf{Stab}_S(x)$. Since $x^{T}=X$, we get $\mathsf{Stab}_S(x)=\{\mathsf{1}\}$. Similarly, $\mathsf{Stab}_{T}(x)=\{\mathsf{1}\}$. We conclude that both $S$ and $T$ act regularly on $X$. Therefore, $T$ has dimension 1. By Fact \ref{fact-frank-finitebyabelian}(2), $T$ has a broad finite-by-abelian normal subgroup. As $T$ is simple, it is abelian, which contradicts Theorem \ref{thm-normalabelian}.

Therefore, $C_{G_0}(T)$ is trivial and $H=T$. By the same reason, $C_G(H)=C_G(T)$ is also trivial. Suppose $\{D_i:i\in I\}$ is another collection of minimal normal subgroups of $G_i$ such that $\{i\in I:D_i\neq H_i\}\in\mathcal{U}$. Then $D_i$ and $H_i$ centralizes each other for all $D_i\neq H_i$. Therefore, $\prod_{i\in I}D_i/\mathcal{U}\leq C_G(H)$, which entails that $\prod_{i\in I}D_i/\mathcal{U}$ is trivial. Hence, $H=\prod_{i\in I}soc(G_i)/\mathcal{U}$.
\end{proof}

Now, we can finish our analysis of higher dimensional cases. We state here a result concerning finite simple groups and a result about counting dimension in ultraproducts of one-dimensional asymptotic classes and SU-rank.

\begin{fact}\label{fact-psl2}(\cite[the Claim in Lemma 5.15]{Elwes}) Let $G(q)$ be a group of Lie type (possibly twisted) over a finite field $\mathbb{F}_q$, with $G\neq \text{PSL}_2(\mathbb{F}_q)$, and let $P(q)$ be a parabolic subgroup of $G(q)$. Then $|G(q):P(q)|>O(q)$.
\end{fact}

\begin{lemma}\label{lem-dimsame}
Let $\mathsf{dim}_1$ and $\mathsf{dim}_2$ be two additive integer-valued dimensions on definable sets of $M$. Suppose for any definable subset $X\subseteq M$ we have $\mathsf{dim}_1(X)=\mathsf{dim}_2(X)$, and that $\mathsf{dim}_1$ is definable, i.e. for any $\emptyset$-definable formula $\varphi(x,y)$ there are $\emptyset$-definable formulas $\psi_1(y),\ldots,\psi_n(y)$ and distinct natural numbers $N_1,\ldots,N_n$ such that $M\models \forall y\left(\exists x\varphi(x,y)\leftrightarrow\bigvee_{1\leq i\leq n}\psi_i(y)\right)$ and $\mathsf{dim}_1(\varphi(M^{|x|},b))=N_i$ if and only if $M\models \psi_i(b)$ for all $b\in M^{|y|}$ and $1\leq i\leq n$. Then $\mathsf{dim}_1$ and $\mathsf{dim}_2$ coincide on all definable sets of $M$. 
\end{lemma}
\begin{proof}
Let $Y\subseteq M^m$ be a definable set defined over  $a$. We induct on $m$. The case $m=1$ is valid by assumption. Suppose $m>1$. Let $\varphi(x_1,\ldots,x_m,a)$ be the formula defining $Y$. As $\mathsf{dim}_1$ is definable, there are distinct natural numbers $K_1,\ldots,K_k$ for some $k\in\mathbb{N}$ and formulas $\psi_1(x_2,\ldots,x_m,a),\ldots,$ $\psi_k(x_2,\ldots,x_m,a)$ such that $\exists x_1\varphi(x_1,M^{m-1},a)=\bigcup_{1\leq i\leq k}\psi_i(M^{m-1},a)$ and $\mathsf{dim}_1(\varphi(M,x_2,\ldots,x_m,a))=N_i$ if and only if $M\models \psi_i(x_2,\ldots,x_m,a)$ for any $x_2,\ldots,x_m\in M^{m-1}$ and $1\leq i\leq k$. Note that $\mathsf{dim}_1(\varphi(M,x_2,\ldots,x_m,a))=\mathsf{dim}_2(\varphi(M,x_2,\ldots,x_m,a))$ for any $x_2,\ldots,x_m\in M^{m-1}$. Hence, we also have $\mathsf{dim}_2(\varphi(M,x_2,\ldots,x_m,a))=N_i$ if and only if $M\models \psi_i(x_2,\ldots,x_m,a)$. Let $Y_i:=\psi_i(M^{m-1},a)$. We may assume $Y_i\neq\emptyset$ for all $1\leq i\leq k$.  By additivity of $\mathsf{dim}_1$ and $\mathsf{dim}_2$, we have $$\mathsf{dim}_j(Y)=\max\{N_i+\mathsf{dim}_j(Y_i):1\leq i\leq k\}$$ for $j\in\{1,2\}$. By induction hypothesis $\mathsf{dim}_1(Y_i)=\mathsf{dim}_2(Y_i)$ for any $1\leq i\leq k$. Therefore, $\mathsf{dim}_1(Y)=\mathsf{dim}_2(Y)$.
\end{proof}
%Notation: We write $A\cong B$ to denote there is a definable isomorphism between $A$ and $B$. When $A,B$ are groups, we write $A\lesssim B$ to denote $A$ can be definably embeds into $B$ in the group structure.
\begin{corollary}\label{cor-sudimc}
Suppose $M=\prod_{i\in I}M_i/\mathcal{U}$ is an infinite ultraproduct of a one-dimensional asymptotic class. Let $\mathsf{dim}_c$ be the counting dimension on $M$, that is, for a definable non-empty set $X=\prod_{i\in I}X_i/\mathcal{U}\subseteq M^n$, define $\mathsf{dim}_c(X):=d$ when there is a real number $r>0$ with $$\lim_{\mathcal{U}}\left(\frac{|X_i|}{|M_i|^d}\right)=r.$$ Then $\mathsf{dim}_c$ coincides with $SU$-rank.
\end{corollary}
\begin{proof}
As $M$ is an infinite ultraproduct of a one-dimensional asymptotic class, $SU(Th(M))=1$ by \cite[Lemma 4.1]{macpherson2008one}. Therefore, for any non-empty definable subset $X\subseteq M$, either $X$ is finite of $SU$-rank 0 or is infinite of $SU$-rank 1. Also note that $\mathsf{dim}_c(X)=0$ if $X$ is finite, and $\mathsf{dim}_c(X)=1$ if $X$ is infinite and $X\subseteq M$. Since $\mathsf{dim}_c$ is definable, we apply Lemma \ref{lem-dimsame} and get the desired result.
\end{proof}

\begin{theorem}\label{thm-dim3}
Let $(G,X)$ be a pseudofinite definably primitive permutation group satisfies the following conditions:
\begin{enumerate}
\item
there is an additive integer-valued dimension on $(G,X)$ such that $\mathsf{dim}(X)=1$ and $\mathsf{dim}(G)\geq 3$;
\item
$G$ and its definable sections satisfy the $\widetilde{\mathfrak{M}_c}$-condition;
\item
$X$ satisfies the (EX)-condition;
\item
$(G,X)$ satisfies the $\widetilde{\mathfrak{M}_s}$-condition.
\end{enumerate}
Let $s(G):=\prod_{i\in I}soc(G_i)/\mathcal{U}$. Then $\mathsf{dim}(G)=3$, $s(G)$ is definable and there is an interpretable pseudofinite field $F$ of dimension 1 such that we can identify $X\cong \textup{PG}_1(F)$, $s(G)\cong \textup{PSL}_2(F)$ and $\textup{PSL}_2(F)\leq G\leq \textup{P}\Gamma\textup{L}_2(F)$.
\end{theorem}
\begin{proof}
Let $H_i:=soc(G_i)$ and $H:=\prod_{i\in I}soc(G_i)/\mathcal{U}$. By the lemmas above, we know that $H=s(G)$ is definable and $H$ is a pseudofinite simple group. By the main theorem of \cite{wilson1995simple}, there is $J\in \mathcal{U}$ such that $H_j$ is a finite Chevalley group of a fixed Lie type and of fixed Lie rank $n$ for all $j\in J$. Take $x=\prod_{i\in I}x_i/\mathcal{U}\in X$. By Lemma \ref{lem-finiteorbit}, the number of orbits of $(H_i)_{x_i}$ is bounded. Hence, we may apply \cite[Theorem 2]{seitz1974small}. It follows that there is $J'\in\mathcal{U}$ such that $J'\subseteq J$ and for all $i\in J'$ the following holds:
there is a parabolic subgroup $P_i$ of $H_i$ and $x_i\in X_i$ such that $(H_i)_{x_i}\leq P_i$. Let $P_i'$ be the maximal parabolic subgroup which contains $P_i$. Let $P:=\prod_{i\in I}P_i'/\mathcal{U}$. By \cite[Lemma 6.2]{stritto2011asymptotic}, $P\lneq H$ is definable in the language of pure groups with parameters in $H$. Note that $P$ is infinite as $H$ is. Also note that $[H:P]=\infty$, since otherwise, $H$ would have a definable normal subgroup of finite index, contradicting that $H$ is a pseudofinite simple group.

%By \cite[Theorem 1.1.1]{ryten2007model}, any family of finite simple groups of Lie type of bounded Lie rank forms an asymptotic class of groups. Hence, the theory of $H$ in the language of pure groups is supersimple of finite $SU$-rank. 

By \cite[Chapter 5]{ryten2007model},
% Theorem 5.2.4; Theorem 5.3.3; Theorem 5.4.6
$H$ is uniformly bi-interpretable with a pseudofinite field $F$ or a pseudofinite difference field $(F,\sigma)$. 
More precisely, there is $J\in\mathcal{U}$ such that one of the following holds:
\begin{itemize}
\item
For all $j\in J$, we have $H_j$ bi-interprets a finite field $\mathbb{F}_j$, and the bi-interpretation is uniform in $j$;
\item
For all $j\in J$, we have $H_j$ bi-interprets a finite difference field of the form $(\mathbb{F}_{2^{2k_i+1}},\mbox{Frob}_{2^{k_i}})$ for some $k_i$, where $\mbox{Frob}_{2^{k_i}}$ is the map $x\mapsto x^{2^{k_i}}$, and the bi-interpretation is uniform in $j$;
\item
For all $j\in J$, we have $H_j$ bi-interprets a finite difference field of the form $(\mathbb{F}_{3^{2k_i+1}},\mbox{Frob}_{3^{k_i}})$ for some $k_i$, where $\mbox{Frob}_{3^{k_i}}$ is the map $x\mapsto x^{3^{k_i}}$, and the bi-interpretation is uniform in $j$.
\end{itemize}
We may assume $F:=\prod_{i\in I}\mathbb{F}_i/\mathcal{U}$ and $(F,\sigma):=\prod_{i\in I}(\mathbb{F}_{2^{2k_i+1}},\mbox{Frob}_{2^{k_i}})/\mathcal{U}$ or $(F,\sigma):=\prod_{i\in I}(\mathbb{F}_{3^{2k_i+1}},\mbox{Frob}_{3^{k_i}})/\mathcal{U}$.

By \cite[Corollary 3.1]{hrushovski1991pseudo} and \cite[Proposition 3.3.19]{ryten2007model}, the theory of $F$ or $(F,\sigma)$ eliminates imaginaries after adding parameters for an elementary submodel. Since both $P$ and $H$ are interpretable in $F$ or in $(F,\sigma)$, so does the right-coset space $P\backslash H$. By elimination of imaginaries, we may suppose that $P\backslash H$ is a definable subset of $F^m$ for some $m$. 

Now we work in $F$ or $(F,\sigma)$. We denote the $SU$-rank in $F$ or $(F,\sigma)$ as $SU_{F}$. And we call a definable set defined in the language of (difference) rings with parameters in $F$ as $F$-definable. Note that $F$ is an ultraproduct of a one-dimensional asymptotic class by \cite{chatzidakis1992definable} for a pure field, and so is $(F,\sigma)$ by \cite[Theorem 3.5.8]{ryten2007model}. Thus, $SU_F(F)=1$.

We claim that for any infinite $F$-definable set $Y\subseteq F^n$, we have $Y$ has positive dimension in $(G,X)$.

Indeed, since $Y$ is infinite, $SU_{F}(Y)\geq 1$. For $1\leq i\leq m$, consider the projection $\pi_i$ of $F^n$ onto the $i^{th}$ co-ordinate. There must be some $i$ such that $\pi_i(Y)$ is an infinite set, i.e., $SU_{F}(\pi_i(Y))\geq 1$. Since $SU_{F}(F)=1$ and $\pi_i(Y)\subseteq F$, we get $SU_{F}(\pi_i(Y))=1$. By the Indecomposability Theorem, there is a definable subgroup $B$ of $F^+$ such that $B\subseteq (\pm\pi_i(Y))^{k}$ for some $k$-fold sum of $\pm\pi_i(Y)$, and finitely many translates of $B$ cover $\pi_i(Y)$. Hence, $SU_{F}(B)=SU_{F}(F^+)=1$, and $B$ has finite index in $F^+$. As $B\subseteq (\pm\pi_i(Y))^{k}$ we get $\mathsf{dim}(B)\leq k\mathsf{dim}(\pi_i(Y))$. 
Therefore, $$\mathsf{dim}(Y)\geq \mathsf{dim}(\pi_i(Y))\geq\frac{1}{k} \mathsf{dim}(B)=\frac{1}{k}\mathsf{dim}(F^+)\geq \frac{1}{k}>0,$$ where the penultimate inequality is by the fact that $H\subseteq F^m$ for some $m\geq 1$ and $\mathsf{dim}(H)\neq 0$, hence $\mathsf{dim}(F)\geq 1$.

Therefore, $\mathsf{dim}(P\backslash H)\geq 1$ and $\mathsf{dim}(P)\geq 1$. Note that $$\mathsf{dim}(P\backslash H)\leq \mathsf{dim}(H_x\backslash H)=\mathsf{dim}(x^H)=1.$$ Hence, $1\leq\mathsf{dim}(P)=\mathsf{dim}(H_x)<\mathsf{dim}(H)$. And we get $\mathsf{dim}(H)\geq 2$. Since $H$ is a definable normal subgroup of $G$, by Lemma \ref{lem-E=H_x}, we get $P=H_x$.

Note that $X$ is in definable bijection with $H_x \backslash H=P\backslash H$. As $P$ is definable in the language of pure groups with parameters in $H$, the action of $H$ on $X$ is interpretable in $H$ itself, hence also interpretable in $F$ or $(F,\sigma)$.

By elimination of imaginaries, we may assume $X$ is definable subset of $F^m$. Consider $SU_{F}(X)$, i.e. $SU_{F}(P\backslash H)$. We claim that $SU_{F}(X)=1$. 

%Indeed, we have  $\mathsf{dim}(F)=1$. Since if we take $Y=X$, then $1=\mathsf{dim}(X)\geq \mathsf{dim}(\pi_i(X))=\mathsf{dim}(F^+)\geq 1$.

Recall that any infinite $F$-definable set has positive dimension. Therefore, any non-algebraic $F$-type can be completed to a $(G,X)$-type of positive dimension.
Take a generic element $\bar{a}=(a_1,\ldots,a_m)\in F^m$ in $X$. Then there is some $i$ such that $tp_{F}(a_i)$ is non-algebraic. Suppose towards a contradiction that $SU_{F}(X)\geq 2$. Then $$2\leq SU_{F}(\bar{a})=SU_{F}(\bar{a}/a_i)+SU_{F}(a_i)=SU_{F}(\bar{a}/a_i)+1.$$ We get $SU_{F}(\bar{a}/a_i)\geq 1$. By the claim above, we have $\mathsf{dim}(\bar{a}/a_i)\geq 1$ and $\mathsf{dim}(a_i)\geq 1$. By the additivity of dimension, $\mathsf{dim}(X)\geq\mathsf{dim}(\bar{a})=\mathsf{dim}(\bar{a}/a_i)+\mathsf{dim}(a_i)\geq 2$, a contradiction. Therefore, $SU_{F}(X)=1$. 

We conclude that $$SU_F(P\backslash H)=SU_F(X)=1=SU_F(F).$$

%We conclude that the theory of $H$ in the pure language of groups is supersimple of finite rank, and $H$ is bi-interpretable with a pseudofinite (difference) field $F$ with $SU_F(F)=1$. Moreover, $H$ has a definable subgroup $P$ with $$SU_F(P\backslash H)=SU_F(X)=1=SU_F(F)$$ and $P$ is a parabolic subgroup of $H$.

Recall that both $F$ and $(F,\sigma)$ are ultraproducts of one-dimensional asymptotic classes. Let $\mathsf{dim}_F$ be the counting dimension on $F$-definable sets. By Corollary \ref{cor-sudimc}, we have $SU_F$ and $\mathsf{dim}_F$ coincide for all $F$-definable sets. Therefore $\mathsf{dim}_F(X)=1$. By definition of $\mathsf{dim}_F$ there is $r\in\mathbb{R}^{>0}$ such that $$\lim_{\mathcal{U}}\left(\frac{|X_i|}{|\mathbb{F}_i|}\right)=r.$$ Therefore, $$\lim_{\mathcal{U}}\left(\frac{|P_i\backslash H_i|}{|\mathbb{F}_i|}\right)=r.$$ By Fact \ref{fact-psl2}, we must have $H\cong \text{PSL}_2(F)$, and $X$ is definably isomorphic to the projective space $\text{PG}_1(F)$.  

Consider $C_G(H)\trianglelefteq G$. It is trivial by Lemma \ref{lem-C(H)}. Therefore, the action of $G$ on $H$ by conjugation is faithful.

Since $H\cong\prod_{i\in I}\text{PSL}_2(\mathbb{F}_{q_i})/\mathcal{U}$ and the largest automorphism group of $\text{PSL}_2(\mathbb{F}_{q_i})$ is $\text{P}\Gamma\text{L}_2(\mathbb{F}_{q_i})$, we get $\text{PSL}_2(F)\leq G\leq \text{P}\Gamma\text{L}_2(F)$ where $\text{P}\Gamma\text{L}_2(F)=\text{PGL}_2(F)\rtimes Aut(F)$.
%, and $Aut(F)$ are those induced by automorphisms of $F$ in the natural way.
\end{proof}
\medskip

\section{Permutation Groups of Infinite $SU$-Rank} \label{sec6}
In this section, we treat the special case when $(G,X)$ is supersimple of infinite $SU$-rank. It is a natural candidate where our classification can be applied. However, the main result of this section is negative. More precisely, we will show that all these groups of dimension greater or equal to 2 will have SU-rank 2 or 3. Hence, there are no interesting infinite SU-rank case.

By Remark \ref{remark-dim}, Remark \ref{remark-mc} and Remark \ref{remark-ms}, we can take the dimension as the coefficient of the leading term of the $SU$-rank and the $\widetilde{\mathfrak{M}_c}$ and $\widetilde{\mathfrak{M}_s}$-conditions always hold in supersimple theories. To apply our classification, it remains to show that when the dimension is greater or equal to 3, $X$ satisfies the (EX)-condition with the assumption of supersimplicity.

\begin{lemma}\label{lem-omegaalpha}
Suppose $(G,X)\in\mathcal{S}$ and its theory is supersimple. Let $A$ be a definable abelian normal subgroup of $G$ and $SU(A)=\omega^\alpha+\beta$ with $\beta<\omega^\alpha$. Then $SU(A)=\omega^\alpha$.
\end{lemma}
\begin{proof}
%Apply \cite[Proposition 5.4.3]{Wagner-Supersimple} to get a type-definable subgroup $D\leq A$ of $SU$-rank $\omega^\alpha$. By \cite[Corollary 4.5.16]{Wagner-Supersimple}, $D$ has a unique minimal type-definable locally connected subgroup $D^c$ which is commensurate with $D$. Hence, $SU(D^c)=SU(D)=\omega^\alpha$. For any $g\in G$, note that both $(D^c)^g$ and $D^c$ are subgroups of $A$ of dimension 1. Therefore, $\mathsf{dim}(D^c\cap (D^c)^g)=1$. We get $\omega^\alpha\leq SU(D^c\cap (D^c)^g)\leq SU(D^c)=\omega^\alpha$. Hence, $(D^c)^g$ is commensurate with $D^c$. As $D^c$ is locally connected, $(D^c)^g=D^c$. We conclude that $D^c$ is normal in $G$. Consider the group $S:=\bigcap_{g\in G}D^g$. Since $D^c\leq S$ and $S$ is a definable subgroup of $D$, we get $SU(S)=SU(D)=SU(D^c)=\omega^\alpha$. Note that $S\leq A$ is also normal in $G$, we get $A=S$, as $A$ is definably minimal normal in $G$. Hence, $SU(X)=SU(A)=\omega^\alpha$.
By \cite[Proposition 5.4.3]{Wagner-Supersimple}, $A$ has a type-definable subgroup $C$ of $SU$-rank $\omega^\alpha$ unique up to commensurability. Since $A$ is normal in $G$, for any $g\in G$ we have $C^g\leq A$. Then $C$ and $C^g$ are commensurable, as $SU(C^g)=\omega^\alpha$ and $C^g\leq A$. By \cite[Lemma 5.5.3] {Wagner-Supersimple}, there is a definable group $D$ with $C\leq D\leq A$ such that $SU(D)=\omega^\alpha$. Since $C\cap C^g\leq D\cap D^g$ and $SU(C\cap C^g)=\omega^\alpha=SU(D)=SU(D^g)$ for any $g\in G$, we get $D$ and $D^g$ are commensurable. By Schlichting's Theorem, we may assume $D$ is normal in $G$. By definably primitivity $D=A$. Therefore, $SU(A)=SU(D)=\omega^\alpha$.
\end{proof}

\begin{corollary}\label{cor2}
Let $(G,X)$ be a pseudofinite definably primitive permutation group whose theory is supersimple. Let $SU(G)=\omega^\alpha n+\gamma$ for some $\gamma<\omega^\alpha$. Suppose $n\geq 3$ and $SU(X)=\omega^\alpha+\beta$ for some $\beta<\omega^\alpha$. Then all the conditions in Theorem \ref{thm-dim3} are satisfied. Hence, there is an interpretable pseudofinite field $F$ such that $X\cong \textup{GL}_1(F)$ and $$\textup{PSL}_2(F)\leq G\leq \textup{P}\Gamma\textup{L}_2(F).$$

Moreover, $G$ is bi-interpretable with $(F,B)$ where $B$ is a group of automorphisms of $F$.
\end{corollary}

\begin{proof}
For any interpretable set $S$ with $SU(S)=\omega^\alpha k+\beta$ for some $\beta<\omega^\alpha$ and $k\geq 0$, we put $\mathsf{dim}(S):=k$. By Remark \ref{remark-dim}, this is an additive integer-valued dimension. Moreover, by supersimplicity $G$ and its definable sections satisfy the $\widetilde{\mathfrak{M}_c}$ and $\widetilde{\mathfrak{M}_s}$-conditions. We only need to check the (EX)-condition. Indeed, we claim that $SU(X)=\omega^\alpha$. Hence, by the Lascar Inequality, $X$ satisfies the (EX)-condition.

\begin{claim}
$SU(X)=\omega^\alpha$.
\end{claim}
\begin{proof}
Let $H:=H_i/\mathcal{U}$, where $H_i$ is a nontrivial minimal normal subgroup of $G_i$. We distinguish two cases: $H$ is abelian and $H$ is non-abelian.

If $H$ is abelian. Then by \cite[Proposition 4.17]{Hempel16} $G$ has a definable finite-by-abelian normal subgroup $A\geq H$. By definably primitivity, $A$ is abelian. %Hence, $A=H$. (why?)
By Lemma \ref{claim 5.1}, $A$ acts regularly on $X$. Since $\mathsf{dim}(X)=1$, we know that $SU(A)=SU(X)=\omega^\alpha+\beta$ for some $\beta<\omega^\alpha$.  By Lemma \ref{lem-omegaalpha}, $SU(A)=\omega^\alpha$. Thus, $SU(X)=\omega^\alpha$.

If $H$ is non-abelian. Then $H$ is definable and $H=T\times T^{g_1}\times\cdots\times T^{g_m}$ for some $m\geq 0$ by Lemma \ref{lemT}. As $T$ is definable and simple, by \cite[Proposition 5.4.9]{Wagner-Supersimple}, $SU(T)=\omega^\alpha k$, for some $k\geq 1$. Therefore, $SU(H)=\omega^\alpha k(m+1)$. Suppose $SU(X)=\omega^\alpha+\beta$ with $\beta<\omega^\alpha$. By the Lascar Inequality, for any $x\in X$, we have $$SU(\mathsf{Stab}_{H}(x))+SU(x^H)\leq SU(H)\leq SU(\mathsf{Stab}_{H}(x))\oplus SU(x^H).$$ As $x^H=X$, we must have $SU(\mathsf{Stab}_{H}(x))=\omega^\alpha (km+k-1)+\gamma$ for some $\gamma<\omega^\alpha$.  Then $$\omega^\alpha k(m+1)= SU(H)\geq SU(\mathsf{Stab}_{H}(x))+SU(x^H)=\omega^\alpha k(m+1)+\beta.$$ We deduce $\beta=0$ and $SU(X)=\omega^\alpha$. 
\end{proof}

By Theorem \ref{thm-dim3} there is an interpretable pseudofinite field $F$ such that $\text{PSL}_2(F)\leq G\leq \text{P}\Gamma\text{L}_2(F)$.

Now we prove that $G$ is bi-interpretable with $(F,B)$ where $B$ is a group of automorphisms of $F$. We identify $G$ with a group between $\text{PSL}_2(F)$ and $\text{P}\Gamma\text{L}_2(F)$ through definable isomorphism. Suppose $(F,B)$ is given and $F=\prod_{i\in I}\mathbb{F}_{q_i}/\mathcal{U}$. As $$\text{P}\Gamma\text{L}_2(\mathbb{F}_{q_i})=\text{PGL}_2(\mathbb{F}_{q_i})\rtimes Gal(\mathbb{F}_{q_i}/\mathbb{F}_{p_i})$$ where $p_i=\mbox{char}(\mathbb{F}_{q_i})$ and $[\text{PGL}_2(\mathbb{F}_{q_i}):\text{PSL}_2(\mathbb{F}_{q_i}))]\leq 2$ for any $i\in I$, we have either $G:=\left(\prod_{i\in I}\text{PSL}_2(\mathbb{F}_{q_i})/\mathcal{U}\right)\rtimes B$ or $G:=\left(\prod_{i\in I}\text{PGL}_2(\mathbb{F}_{q_i})/\mathcal{U}\right)\rtimes B$. Clearly $G$ is interpretable in $(F,B)$ in both cases.

Suppose $G=H\rtimes B$ is given, where $B\leq Aut(F)$. By the argument before, $G$ interprets $F$.  Let $\varphi(g,x,y)$ be the formula expressing: $x,y\in F$ and 
$$\left[\begin{pmatrix} 1 &  x \\ 0 & 1 \end{pmatrix}\right]^g =
\left[\begin{pmatrix}1 &y\\0&1
\end{pmatrix}\right],$$ where $\left[\begin{pmatrix} a &  b\\ c & d \end{pmatrix}\right]$ denotes the coset $\begin{pmatrix} a &  b \\ c & d \end{pmatrix}F^\times$ in $\text{PGL}_2(F)$. Then $\varphi(g,F,F)$ is the graph of a partial function. Let $\xi(g)$ be the formula expressing that $\varphi(g,F,F)$ is the graph of a field automorphism of $F$. Define $\phi(g,x,y):=\varphi(g,x,y)\land \xi(g)$  and $\sim$ be the equivalence relation on $G\times F\times F$ defined as $(g,x,y)\sim (g',x',y')$ if and only if $x=x',y=y'$ and $\varphi(g,F,F)=\varphi(g',F,F)$. Then $\phi(G,F,F)/\sim$ is a group of automorphisms of $F$ containing $B$. We need to show that $\phi(G,F,F)/\sim$ contains no other automorphisms. Note that $\xi(G)$ defines a subgroup of $G$. Then $\xi(G)\cap H=\xi(H)\leq G$. Let $\sim_{H}$ be the equivalence relation such that $g\sim_{H} g'$ if and only if $\varphi(g,F,F)=\varphi(g',F,F)$. Then $\xi(H)/\sim_{H}$ is a group of automorphism of $F$. As $H$  and $\xi(H)$ are interpretable in $F$, so does $\xi(H)/\sim_{H}$. We conclude $\xi(H)/\sim_{H}$ is trivial by the fact that a pure field can only interpret the trivial group of field-automorphisms of itself. Therefore $B=\phi(G,F,F)/\sim$.
%Let $g\in P\Gamma L_2(F)$. We can express by a formula $\varphi(g)$ the following:\begin{enumerate}
%\item
%for all $x\in\mathbb{F}$, there exists one and only one $y\in \mathbb{F}$, such that $\begin{pmatrix}1 &y\\0&y+1
%\end{pmatrix}$ belongs to the coset
%$[\begin{pmatrix} 1 &  x \\ 0 & x+1 \end{pmatrix}]^g$, where $[\begin{pmatrix} 1 &  x \\ 0 & x+1 \end{pmatrix}]$ denotes the coset $\begin{pmatrix} 1 &  x \\ 0 & x+1 \end{pmatrix}F$;
%\item
%let $f_{g}$ be the function whose definition and existence is given by the previous condition. Then $f_g$ is an automorphism of the field $F$.
%\end{enumerate}
%Note that if $g\in P\Gamma L_2(F)$ is induced by an automorphism of $F$, then $\varphi(g)$ holds. Hence, the formula $$\psi(z):=\exists y\in G^*~\land~\exists x\in \text{PGL}_2(F)~(z=xy\land \varphi(xy))$$ defines a group $B$ of automorphisms of $F$. Clearly, $G$ is bi-interpretable with $(F,B)$.
\end{proof}

In the following, we will exclude the possibility that $B$ is infinite. This is due to the fact that any structure that expands a pseudofinite field with a ``logarithmically small" infinite set will have the strict order property, hence, its theory will not be simple. Since an infinite definable set of automorphisms of a pseudofinite field is always ``logarithmically small" compared to the size of the field, $B$ must be finite by simplicity, hence trivial.

\begin{fact}(Folklore, see also \cite[Theorem 30]{ZouDifference})\label{thm-sop}
Let $F=\prod_{i\in I}\mathbb{F}_{{p_i}^{n_i}}/\mathcal{U}$ be a pseudofinite field and $A=\prod_{i\in I}A_i/\mathcal{U}$ an infinite pseudofinite subset of $F$. Suppose there is a constant natural number $C$ such that $|A_i|\leq Cn_i$ for any $i\in I$. Then the theory of $(F,A)$ has the strict order property.
\end{fact}

\begin{corollary}\label{thmnonsimple}
Suppose $(F,B)=\prod_{i\in I}(\mathbb{F}_{p_i^{n_i}}, B_i)/\mathcal{U}$ is a pseudofinite structure with $F$ a field and $B$ an infinite set of automorphisms of $F$. Then the theory of $(F,B)$ is not simple.
\end{corollary}
\begin{proof}
Take $a_i$ a generator of the multiplicative group of $\mathbb{F}_{p_i^{n_i}}$. Define $A_i=a_i^{B_i}$. As $a_i$ is the generator and all $B_i$ are powers of the Frobenius, we have $|A_i|=|B_i|\leq n_i$. Let $A=\prod_{i\in I}A_i/\mathcal{U}$. Then we can apply Fact \ref{thm-sop} to $(F,A)$ and get the desired result.
\end{proof}

Combing the results above, we have the following conclusion.
\begin{theorem}\label{thm-noInfiniteRank}
Let $(G,X)$ be a pseudofinite definably primitive permutation group whose theory is supersimple. Let $SU(G)=\omega^\alpha n+\gamma$ for some $\gamma<\omega^\alpha$ and $n\geq 1$. Suppose $SU(X)=\omega^\alpha+\beta$ for some $\beta<\omega^\alpha$. Then one of the following holds:
\begin{enumerate}
\item
$SU(G)=\omega^\alpha+\gamma$, and there is a definable, divisible torsion-free or elementary abelian subgroup $A$ of $SU$-rank $\omega^\alpha$ which acts regularly on $X$. 
\item
$SU(G)=2$, and there is an interpretable pseudofinite field $F$ of $SU$-rank $1$ such that $G\cong F^+\rtimes D$ where $D$ has finite index in $F^\times$.
\item
$SU(G)=3$, and there is an interpretable pseudofinite field $F$ of $SU$-rank $1$ such that $G\cong \textup{PSL}_2(F)$ or $G\cong \textup{PGL}_2(F)$.
\end{enumerate}

\end{theorem}
\begin{proof}
Let $\mathsf{dim}$ be defined as the coefficient of $\omega^\alpha$.

When $n=1$, we apply Theorem \ref{thm-dim1} and get a definable normal abelian subgroup $A$ of $SU$-rank greater than or equal to $\omega^\alpha$. By Lemma \ref{lem-omegaalpha}, we have $SU(A)=\omega^\alpha$.
%By \cite[Proposition 5.4.3]{Wagner-Supersimple}, $B$ has a type-definable subgroup $C$ of $SU$-rank $\omega^\alpha$ unique up to commensurability. Hence, $C$ and $C^g$ are commensurable for any $g\in G$. By \cite[Lemma 5.5.3] {Wagner-Supersimple}, there is a definable group $A$ with $C\leq A\leq B$ such that $SU(A)=\omega^\alpha$. Therefore, $A$ and $A^g$ are commensurable for any $g\in G$. By Schlichting's Theorem, we may assume $A$ is normal in $G$. By definably primitivity $B=A\simeq X$.
%If $B\leq A$ is non-trivial and normal in $G$. Take $b\in B$, then by the Indecomposability Theorem, there is some definable  group $C\leq (b^B)^k$ for some $k\in\mathbb{N}$, which is also normal in $G$. By definably primitivity, $A=C=B$. Hence, $A=soc(G)$. By Claim \ref{claim omega}, we have $SU(A)=\omega^\alpha$.

If $n=2$, then by Theorem \ref{thm-dim2}, there is an interpretable pseudofinite field $F$ of dimension 1 such that $G_x$ induces a group of automorphisms $B$ on $F$. By Corollary \ref{thmnonsimple}, we know that $B$ must be finite. Then by Corollary \ref{cor1}, $G_x$ embeds into $F^\times$ and $B$ is trivial. Since the $SU$-rank of $F^\times$ is a monomial, and $\mathsf{dim}(F)=\mathsf{dim}(G_x)=1$, we get $SU(G_x)=SU(F^\times)=\omega^{\alpha}$. Therefore, $G_x$ has finite index in $F^\times$. Suppose $[F^\times:G_x]=k$. Consider $(F^\times)^k=\{g^k:g\in F^\times\}$. As $F^\times=\prod_{i\in I}F_i^\times/\mathcal{U}$, there is $J\in \mathcal{U}$ such that $F_i$ is cyclic for all $i\in J$ and $(F_i^\times)^k$ is the unique subgroup of index $k$. Therefore, $(F^\times)^k$ is also the unique definable subgroup of index $k$ of $F^\times$. Thus, $G_x=(F^\times)^k$. Now $(G,X)$ is definable in $F$, so $(G,X)$ is supersimple of $SU$-rank 2.

If $n\geq 3$, then by Corollary \ref{cor2}, $(G,X)$ is bi-interpretable with a pseudofinite field $F$ together with a group of automorphisms $B$. By Corollary \ref{thmnonsimple}, $B$ is finite, hence is trivial by Lemma \ref{lem-Btrivial}. Therefore, $\text{PSL}_2(F)\leq G\leq \text{PGL}_2(F)$. For any finite field $\mathbb{F}_q$, we have $[\text{PGL}_2(\mathbb{F}_q):\text{PSL}_2(\mathbb{F}_q)]\leq 2$. Hence, either $G\cong \text{PSL}_2(F)$ or $G\cong \text{PGL}_2(F)$.
\end{proof}

\subsection*{Acknowledgements}
This author is supported by the China Scholarship Council and partially supported by ValCoMo (ANR-13-BS01-0006). 
This paper is based on the author's main PhD thesis project. The author wants to thank her supervisor Frank Wagner for suggesting this interesting yet challenging topic and for his patient guidance and enormous help during the work on this project. She is very grateful to the referee for pointing out problems in the previous version and giving a lot of useful comments and suggestions. She also wants to thank Dugald Macpherson for plenty of helpful email exchanges.

%%%%%%%%%%% To ease editing, use normal size for the references:

\end{document}